\documentclass[12pt,twoside]{amsart}
\usepackage{amsthm,amsfonts,amssymb,amscd,euscript}

\usepackage[matrix,arrow]{xy}

\author{Florin Ambro} 
\address{Institute of Mathematics ``Simion Stoilow'' of the Romanian Academy\\
P.O. BOX 1-764, RO-014700 Bucharest\\ 
Romania.}
\email{florin.ambro@imar.ro}

\newcommand{\isoto}{{\overset{\sim}{\rightarrow}}}

\newcommand{\C}{{\mathbb C}}
\newcommand{\Q}{{\mathbb Q}}
\newcommand{\Z}{{\mathbb Z}}
\newcommand{\N}{{\mathbb N}}
\newcommand{\R}{{\mathbb R}}

\newcommand{\bA}{{\mathbb A}} 

\newcommand{\cA}{{\mathcal A}}
\newcommand{\cB}{{\mathcal B}}
\newcommand{\cC}{{\mathcal C}}

\newcommand{\cF}{{\mathcal F}}

\newcommand{\cL}{{\mathcal L}}

\newcommand{\cO}{{\mathcal O}}

\newcommand{\cT}{{\mathcal T}}

\newcommand{\cX}{{\mathcal X}}

\newcommand{\bF}{{\mathbf F}}

\newcommand{\bH}{{\mathbb H}}  

\newcommand{\fm}{{\mathfrak m}}

\newcommand{\Char}{\operatorname{char}}
\newcommand{\Cl}{\operatorname{Cl}}

\newcommand{\Der}{\operatorname{Der}}

\newcommand{\dv}{\operatorname{div}}
\newcommand{\emb}{\operatorname{emb}}

\newcommand{\Hom}{\operatorname{Hom}}

\newcommand{\mult}{\operatorname{mult}}

\newcommand{\ord}{\operatorname{ord}}

\newcommand{\Sing}{\operatorname{Sing}}
\newcommand{\Spec}{\operatorname{Spec}}
\newcommand{\Supp}{\operatorname{Supp}}

\theoremstyle{plain}
\newtheorem{thm}{Theorem}[section]

\newtheorem{lem}[thm]{Lemma}
\newtheorem{cor}[thm]{Corollary}

\newtheorem{prop}[thm]{Proposition}

\theoremstyle{definition}

\newtheorem{exmp}[thm]{Example}

\newtheorem{rem}[thm]{Remark}

\newtheorem{ack}{Acknowledgments}

\begin{document}

\bibliographystyle{amsalpha+}
\title{Cyclic covers and toroidal embeddings}
\maketitle

\begin{abstract} 
We present an improved version of the cyclic covering trick, which works
inside the category of toroidal embeddings.
\end{abstract} 



\footnotetext[1]
{
This work was supported by a grant of the 
Romanian National Authority for Scientific Research, 
CNCS - UEFISCDI, project number PN-II-RU-TE-2011-3-0097.
}


\section*{Introduction}


Cyclic covers are a useful tool in algebraic geometry. The simplest example is 
the field extension 
$$
K\subset K(\sqrt[n]{\varphi})
$$
obtained by adjoining to a field the root of an element. For example,
the equation $t^n-\varphi \prod_i z_i^{m_i}$ over $K$ simplifies to 
$t^n-\prod_i z_i^{m_i}$ over $K(\sqrt[n]{\varphi})$.

In classical algebraic geometry, cyclic covers were used to construct new 
examples from known ones. Given a complex projective variety $X$, a torsion
line bundle $L$ over $X$ induces canonically an \'etale cyclic Galois covering $\pi\colon X'\to X$
such that $\pi^*L$ becomes trivial. If $L$ has torsion index $r$ 
and $s\in \Gamma(X,L^r)$ is nowhere zero, the covering can be constructed as the 
$r$-th root of $s$ (as in the function field case, the pullback of $s$ becomes an $r$-th 
power of a section of the pullback of $L$). We may denote it by $\pi\colon X[\sqrt[r]{s}]\to X$.
The isomorphism class of $\pi$ does not depend on the choice of $s$, since $X$ being 
compact, every two nowhere zero sections differ by a non-zero constant. 

Many invariants of $X'$ can be read off those
of $X$, but with coefficients in negative powers of $L$. For example,
$$
\pi_*\Omega^p_{X'}=\oplus_{i=0}^{r-1}\Omega^p_X\otimes L^{-i}.
$$
So one may construct manifolds with prescribed invariants by taking roots of
torsion line bundles on known manifolds. The process may be reversed: 
known statements on the invariants of $X'$ translate into similar statements on $X$, 
twisted by negative powers of $L$.
For example, the K\"ahler differential of $X'$ decomposes into integrable flat connections
on $L^{-i}$, so that $\oplus_{i=0}^{r-1}\Omega^\bullet_X(L^{-i})$ is the 
Hodge complex $\pi_*\Omega^\bullet_{X'}$ on $X$. In particular, the $E_1$-degeneration for 
$(\Omega^\bullet_{X'},F)$ translates into the $E_1$-degeneration for 
$(\Omega^\bullet_{X'}(L^{-i}),F)$, for every $i$.
This exchange of information between the total and base space of a cyclic cover is called
the {\em cyclic covering trick} (cf. sections 1 and 2 of~\cite{EVlect} for example).

The range of applications of the cyclic covering trick extends dramatically if 
$s$ is allowed to have zeros. In this case $s$ is a non-zero global section of
the $n$-th power of some line bundle $L$ on $X$. The $n$-th root of $s$ is defined 
just as above. We obtain for example the same formula
$$
\pi_* \cO_{ X[\sqrt[n]{s}] } = \oplus_{i=0}^{n-1} L^{-i}.
$$
The morphism $\pi$ is still cyclic Galois and flat, but it ramifies over the zero locus
of $s$. The total space $X[\sqrt[n]{s}]$ may be disconnected (even if $s$ vanishes nowhere), 
it may have several irreducible components, and it may even have singularites over 
the zero locus of $s$. These singularities are 
partially resolved by the normalization $\bar{X}[\sqrt[n]{s}] \to X[\sqrt[n]{s}]$. 
The induced morphism $\bar{\pi}\colon \bar{X}[\sqrt[n]{s}] \to X$ is cyclic Galois and flat, 
and one computes
$$
\bar{\pi}_* \cO_{ \bar{X}[\sqrt[n]{s}] } = \oplus_{i=0}^{n-1} L^{-i}(\lfloor \frac{i}{n} Z(s)\rfloor).
$$
Here $Z(s)$ is the effective Cartier divisor cut out by $s$, and the round down
of the $\Q$-divisor $\frac{i}{n} Z(s)$ is defined componentwise. If $\Supp Z(s)$ has no singularities, then 
$\bar{X}[\sqrt[n]{s}]$ has no singularities. Differential forms or vector fields
on $\bar{X}[\sqrt[n]{s}]$ are computed in terms of $X,L$, and the $\Q$-divisor $\frac{1}{n} Z(s)$.
For example
$$
\bar{\pi}_* \Omega^p_{ \bar{X}[\sqrt[n]{s}] } = \oplus_{i=0}^{n-1} \Omega^p_X(\log 
\Supp\{\frac{i}{n} Z(s) \})\otimes L^{-i}(\lfloor \frac{i}{n} Z(s)\rfloor),
$$
where $\{\frac{i}{n} Z(s) \}$ is the fractional part of the $\Q$-divisor $\frac{i}{n} Z(s)$,
defined componentwise, and for a reduced divisor $\Sigma$ on $X$, $\Omega^p_X(\log \Sigma)$
denotes the sheaf of differential $p$-forms $\omega$ such that both $\omega$ and $d\omega$
are regular outside $\Sigma$, and have at most logarithmic poles along the components of $\Sigma$. 
If the singularities of $\Supp Z(s)$ are at most simple normal crossing, then 
$\bar{X}[\sqrt[n]{s}]$ has at most quotient singularities, and if $Y\to \bar{X}[\sqrt[n]{s}]$ is
a desingularization, with $\nu\colon Y\to X$ the induced generically finite morphism,
then
$$
\nu_* \Omega^p_Y = \oplus_{i=0}^{n-1} \Omega^p_X(\log 
\Supp\{\frac{i}{n} Z(s) \})\otimes L^{-i}(\lfloor \frac{i}{n} Z(s)\rfloor).
$$
This formula is behind the vanishing theorems
used in birational classification (see~\cite{EVlect,Kol95,Kol97}). Statements on divisors of the form
$K_X+\sum_j b_j E_j+T$, with $X$ nonsingular, $\sum_j E_j$ simple normal crossing,
$b_j\in [0,1]$, and $T$ a torsion $\Q$-divisor, are reduced to similar statements
on $Y$ with $b_j\in \{0,1\}$ and $T=0$.

Cyclic covers also appear in semistable reduction~\cite{KKMS}. In its simplest form, a complex
projective family over the unit disc $f\colon \cX\to \Delta$ has nonsingular general fibers 
$\cX_t \ (t\ne 0)$, while the special fiber $\cX_0$ is locally cut out by monomials 
$\prod_{i=1}^d z_i^{m_i}\ (m_i\in \N)$ with respect to local coordinates $z_1,\ldots,z_d$.
The family is semistable if moreover $\cX_0$ is reduced. If we base change 
with $\sqrt[n]{t}$ (with $n$ divisible by all multiplicities $m_i$),
and normalize $\tilde{\cX}\to \cX\times_\Delta \tilde{\Delta}$, the new family 
$\tilde{\cX}\to \tilde{\Delta}$ has reduced special fiber $\tilde{\cX}_0$, and 
$\tilde{\cX}\setminus \tilde{\cX}_0\subset \tilde{\cX}$ is a {\em quasi-smooth toroidal embedding}.
If the irreducible components of $\cX_0$ are nonsingular, the toroidal embedding
is also strict and $\tilde{\cX}$ admits a combinatorial desingularization.
An equivalent description of $\tilde{\cX}$ is the normalization of the 
$n$-th root of $f$, viewed as a holomorphic function on $\cX$.
Therefore the local computations of~\cite{KKMS} give in fact the following statement:
if $X$ is complex manifold, and $0\ne s\in \Gamma(X,L^n)$ is such that 
$\Sigma=\Supp Z(s)$ is a normal crossing divisor, then $\bar{X}[\sqrt[n]{s}]\setminus 
\bar{\pi}^{-1}(\Sigma)\subset \bar{X}[\sqrt[n]{s}]$ is a quasi-smooth toroidal embedding, 
and $\bar{\pi}$ is a toroidal morphism.

Cyclic covers are used to classify the singularities that appear
in the birational classification of complex manifolds. Such singularities $P\in X$
are normal, and the canonical Weil divisor $K_X$ is a torsion element of
$\Cl(\cO_{X,P})$. If $r$ is the torsion index, there exists a rational function 
$\varphi\in \C(X)^*$ such that $rK_X=\dv(\varphi)$. The normalization of 
$X$ in the Kummer extension $\C(X)\to \C(X)(\sqrt[r]{\varphi})$ becomes
a cyclic cover $P'\in X'\stackrel{\pi}{\to} P\in X$. It is called the 
{\em index one cover} of $P\in X$, since being \'etale in codimension one,
$K_{X'}=\pi^*K_X\sim 0$. 
The known method to classify $P\in X$ is to first classify the index one cover,
and then understand all possible actions of cyclic groups (see~\cite{YPG87}).

We have discussed so far roots of rational functions, (normalized) roots of multi sections 
of line bundles, and index one covers of torsion $\Q$-divisors on normal varieties. 
This paper gives a unified treatment of all these concepts, based on {\em normalized
roots of rational functions on normal varieties}. 
Moreover, we show that the cyclic covering trick can be performed inside the category of 
quasi-smooth toroidal embeddings. In order to prove vanishing theorems, we no longer have to
assume that the base is nonsingular, or to resolve the singularities of the total space of the covering.

To state the main results of this paper, let $k$ be an algebraically closed field. 

\begin{thm}\label{m1}
Let $X/k$ be an normal algebraic variety. Let $\varphi$ be an invertible rational function on $X$,
let $n$ be a positive integer such that  $\Char k \nmid n$. Denote 
$D=\frac{1}{n}\dv(\varphi)$, so that $D$ is a $\Q$-Weil divisor on $X$ with $nD\sim 0$. 
Let $\pi\colon Y\to X$ be the normalization of $X$
with respect to the ring extension 
$$
k(X)\to \frac{k(X)[T]}{(T^n-\varphi)}.
$$ 
The right hand side is a product of fields (the function fields of the irreducible
components of $Y$), and $Y$ identifies with the disjoint union of the 
normalization of $X$ in each field. By construction, $Y/k$ is a normal algebraic variety 
(possibly disconnected). 
 
\begin{itemize}
\item[a)] The class of $T$ becomes an invertible rational function $\psi$ 
on $Y$ such that $\psi^n=\pi^*\varphi$. We have $\pi^*D=\dv(\psi)$ and
$$
\pi_*\cO_Y=\oplus_{i=0}^{n-1}\cO_X(\lfloor iD\rfloor)\cdot \psi^i.
$$
The morphism $\pi$ is \'etale exactly over $X\setminus \Supp\{D\}$, where $\{D\}$ is the fractional
part of the $\Q$-divisor $D$, defined componentwise. 
It is flat if and only if the Weil divisors $\lfloor iD\rfloor \ (0<i<n)$ are Cartier.

\item[b)] Suppose $U\subseteq X$ is a quasi-smooth toroidal embedding and $D|_U$ has integer coefficients.
Then $\pi^{-1}(U)\subseteq Y$ is a quasi-smooth toroidal embedding, and $\pi$ is a toroidal morphism.
Denote $\Sigma_X=X\setminus U$ and $\Sigma_Y=Y\setminus \pi^{-1}(U)$.
Then $\pi^*\tilde{\Omega}^p_{X/k}(\log \Sigma_X) \isoto \tilde{\Omega}^p_{Y/k}(\log \Sigma_Y)$, 
and by the projection formula
$$
\pi_*\tilde{\Omega}^p_{Y/k}(\log \Sigma_Y)=\tilde{\Omega}^p_{X/k}(\log \Sigma_X)\otimes \pi_*\cO_Y.
$$ 
\end{itemize}
\end{thm}

\begin{thm}\label{m2}
Suppose $\Char k =0$. Let $U\subseteq X$ and $U'\subseteq X'$ be toroidal embeddings 
over $k$, let $\mu\colon X'\to X$ be a proper morphism which induces 
an isomorphism $U'\isoto U$. Denote $\Sigma_X=X\setminus U$ and $\Sigma_{X'}=X'\setminus \Sigma_{X'}$.
Then 
$$
R^q\mu_*\tilde{\Omega}^p_{X'/k}(\log \Sigma_{X'}) =
\left\{ \begin{array}{lll}
\tilde{\Omega}^p_{X/k}(\log\Sigma_X)                      & q=0   \\
0              							                            &  q\ne 0
\end{array} \right.
$$
\end{thm}

Theorem~\ref{m1} generalizes the known cyclic covering trick (see~\cite[Section 3]{EVlect}, especially
3.3-3.11), where the base space $X$ is assumed nonsingular, to the case when $X$ is just normal.
Statement a) is elementary.
The toroidal part of b) is implicit in~\cite{KKMS} if $(X,\Sigma_X)$ is log smooth, as already 
mentioned. The general case (Theorem~\ref{tcr}) is proved by reduction to the following fact: the 
normalized root of a toric variety with respect to a torus character consists of several isomorphic
copies of a toric morphism (Proposition~\ref{nru}).
The sheaf $\tilde{\Omega}^p_{X/k}(\log \Sigma_X)$ consists of the rational $p$-forms $\omega$
of $X$ such that both $\omega,d\omega$ are regular near the prime divisors of $X$ outside $\Sigma_X$,
and have at most simple poles at the prime components of $\Sigma_X$. It is called the sheaf of logarithmic 
$p$-forms of $(X/k,\Sigma_X)$, in the sense of Zariski-Steenbrink. It is constructed by 
ignoring closed subsets of $X$ of codimension at least two, so in general it is singular.
But if $X\setminus \Sigma_X\subset X$ is a toroidal embedding, it is locally free~\cite{Ste76,Dan78}.
If $X$ is nonsingular and $\Sigma_X$ is a normal crossing divisor, this sheaf
coincides with the sheaf of logarithmic forms $\Omega^p_{X/k}(\log\Sigma_X)$ in the sense of 
Deligne (see~\cite{EVlect} for the algebraic version, with $\Sigma_X$ assumed simple normal crossing). 
If $\Sigma_X=0$, then $\tilde{\Omega}^p_{X/k}=\tilde{\Omega}^p_{X/k}(\log 0)$ is the double dual of
the usual sheaf of K\"ahler differentials $\Omega^p_{X/k}$.
We note that differential forms or vector fields on $Y$ can be computed without the 
toroidal assumption (Lemma~\ref{dvc}).

Theorem~\ref{m2} is the invariance of the logarithmic sheaves under different toroidal embeddings,
proved by Esnault and Viehweg~\cite[Lemma 1.5]{RCII} in the case when $X$ is projective nonsingular 
and $\Sigma_X$ has normal crossings. One corollary is that
$$
H^q(X,\tilde{\Omega}^p_{X/k}(\log\Sigma_X))\to H^q(X',\tilde{\Omega}^p_{X'/k}(\log\Sigma_{X'}))
$$ 
is an isomorphism for every $p,q$. If $X$ is proper and $(X,X\setminus U)$ is log smooth, the
corollary follows from the $E_1$-degeneration of the spectral
sequence induced in hypercohomology by the logarithmic De Rham complex endowed with the
naive filtration (Deligne~\cite{Del71}).
If $X$ is projective and $X\setminus U$ is a simple normal crossing divisor, 
Esnault-Viehweg~\cite[Lemma 1.5]{RCII} proved that the corollary implies Theorem~\ref{m2}.
We use the same idea, combined with a result of Bierstone-Milman~\cite{BM11A}, in order
to compactify strict log smooth toroidal embeddings (Corollary~\ref{SNCemb}).

The normalization of roots of multi sections of line bundles on normal varieties,
and the index one covers torsion $\Q$-divisors on normal varieties, are both 
examples of normalized roots of rational functions. In practice, index one covers 
are most useful. They preserve irreducibility, so one can work in the classical 
setting of function fields. Their drawback is that they do not commute with base change
to open subsets. For this reason, at least for proofs, we need to consider 
normalized roots of rational functions, which commute with \'etale base change.

\begin{ack} Part of this work was done while visiting NCTS (Mathematics Division, Taipei Office)
and IMS (National University of Singapore), and I am grateful to Jungkai Alfred Chen
and De-Qi Zhang for hospitality. I would also like to thank Lucian B\u{a}descu, 
Steven Lu and Rita Pardini for useful discussions.
\end{ack}


\section{Preliminaries}


We consider varieties (reduced, possibly reducible), defined over an algebraically closed field $k$.


\subsection{Zariski-Steenbrink differentials~\cite{Zar69,Ste76}}


Let $X/k$ be a normal variety. Its connected
components coincide with its irreducible components. Let $k(X)$ be the ring of rational
functions of $X$, consisting of functions regular on a dense open subset of $X$. It
identifies with the product of the function fields of the irreducible components of $X$.
A rational function is invertible if and only if it is not zero on each irreducible component.

Let $\omega\in \Omega^p_{k(X)/k}$ be a rational differential $p$-form. Let $E\subset X$ 
be a prime divisor. The form $\omega$ is regular at $E$ if it can be written as a sum of
forms $f_0 \cdot df_1\wedge \cdots \wedge df_p$, with $f_i\in \cO_{X,E}$. The form $\omega$
has at most a logarithmic pole at $E$ if both $\omega$ and $d\omega$ have at most a simple
pole at $E$.
If $t$ is a local parameter at the generic point of $E$, this is equivalent to the existence of 
a decomposition
$
\omega=\frac{dt}{t}\wedge \omega^{p-1}+\omega^p,
$
with $\omega^{p-1},\omega^p$ rational differentials regular at $E$.
If $\bar{E}\to E$ is the normalization, the restriction $\omega^{p-1}|_{\bar{E}}$ is 
independent of the decomposition, called the residue of $\omega$ at $E$. The 
residue is zero if and only if $\omega$ is regular at $E$.

For an open subset $U\subseteq X$, let $\Gamma(U,\tilde{\Omega}^p_{X/k})$ consist
of the rational differential $p$-forms which are regular at each prime divisor of $X$
which intersects $U$. This defines a coherent $\cO_X$-module $\tilde{\Omega}^p_{X/k}$,
called the {\em sheaf of $p$-differential forms of $X/k$, in the sense of Zariski-Steenbrink}. If 
$j\colon X^0\subset X$ is the nonsingular locus of $X$, 
$\tilde{\Omega}^p_{X/k}=j_*(\Omega^p_{X^0/k})$.
If $D$ is a Weil divisor on $X$, let $\Gamma(U,\tilde{\Omega}^p_{X/k}(D))$ consist
of the rational differential $p$-forms $\omega$ such that $t_E^{\mult_E D}\omega$
is regular at $E$, for every prime divisor $E$, with local parameter $t_E$. This defines 
a coherent $\cO_X$-module $\tilde{\Omega}^p_{X/k}(D)$. We have $\tilde{\Omega}^p_{X/k}(D)=
j_*(\Omega^p_{X^0/k}\otimes\cO_{X^0}(D|_{X^0}))$.

Let $\Sigma$ be a reduced Weil divisor on $X$, or equivalently, a finite set of prime 
divisors on $X$. For an open subset $U\subseteq X$, let $\Gamma(U,\tilde{\Omega}^p_{X/k}(\log \Sigma))$ consist of the rational differential $p$-forms $\omega$ such that for every prime divisor 
$E$ which intersects $U$, $\omega$ is regular (has at most a logarithmic pole) at $E$ if 
$E\notin \Sigma$ ($E\in \Sigma$). 
This defines a coherent $\cO_X$-module $\tilde{\Omega}^p_{X/k}(\log\Sigma)$,
called the {\em sheaf of logarithmic $p$-differential forms of $(X/k,\Sigma)$, in the sense 
of Zariski-Steenbrink}. For a Weil divisor $D$ on $X$, we can similarly define 
$\tilde{\Omega}^p_{X/k}(\log\Sigma)(D)$.

The tangent sheaf $\cT_{X/k}$ is already $S_2$-saturated, since $X$ is normal:
a derivation $\theta$ of $k(X)/k$ is regular on an open subset $U\subseteq X$ if
and only if it is regular at each prime of $X$ which intersects $U$.
For an open subset $U\subseteq X$, let $\Gamma(U,\tilde{\cT}_{X/k}(-\log \Sigma))$ 
consist of the derivations $\theta\in \Gamma(U,\cT_{X/k})$ such that for every prime 
divisor $E\in \Sigma$ which intersects $U$, $\theta$ preserves the maximal ideal $\fm_{X,E}$.
This defines a coherent $\cO_X$-module $\tilde{\cT}_{X/k}(-\log\Sigma)$,
called the {\em sheaf of logarithmic derivations of $(X,\Sigma)$, in the sense 
of Zariski-Steenbrink}. For a Weil divisor $D$ on $X$, we can similarly define 
$\tilde{\cT}_{X/k}(D),\tilde{\cT}_{X/k}(-\log\Sigma)(D)$.

\begin{lem}\label{eBC}
Let $u\colon X'\to X$ be an \'etale morphism. In particular, $X'$ is normal.
Let $\Sigma$ be a reduced Weil divisor on $X$, denote $\Sigma'=u^{-1}(\Sigma)$.
Then we have base change isomorphisms
$$
u^*\tilde{\Omega}^p_{X/k}\isoto \tilde{\Omega}^p_{X'/k} \text{ and }
u^*\tilde{\Omega}^p_{X/k}(\log \Sigma) \isoto  \tilde{\Omega}^p_{X'/k}(\log \Sigma').
$$
Similar statements hold for twists with Weil divisors.
\end{lem}

\begin{proof} Denote $U=X\setminus \Sing X$ and $U'=X'\setminus \Sing X'$.
Since $u$ is smooth, we have $U'=u^{-1}(U)$. We obtain a cartesian diagram
\[ 
\xymatrix{
 U' \ar[d]_{i_{U'}}  \ar[r]^v & U  \ar[d]^{i_U}  \\
 X'    \ar[r]_u        &  X      
} \]
Since $u$ is flat, the base change homomorphism
$u^*{i_U}_*\Omega^p_{U/k}\to {{i_{U'}}_*}v^*\Omega^p_{U/k}$ is an isomorphism.
Since $v$ is smooth, $v^*\Omega^p_{U/k}\to \Omega^p_{U'/k}$ is an isomorphism.
Therefore the base change isomorphism becomes 
$u^*\tilde{\Omega}^p_{X/k}\isoto \tilde{\Omega}^p_{X'/k}$.

Since $u$ is \'etale, we have $\Sigma'=u^{-1}(\Sigma)$, and $\Sigma'\to \Sigma$ is \'etale.
In particular, $\Sing \Sigma'=u^{-1}(\Sing \Sigma)$.
Denote $V=X\setminus (\Sing X\cup \Sing \Sigma)$ and $V'=X'\setminus (\Sing X'\cup \Sing \Sigma')$.
We obtain a cartesian diagram
\[ 
\xymatrix{
 V' \ar[d]_{i_{V'}}  \ar[r]^v & V  \ar[d]^{i_V}  \\
 X'    \ar[r]_u        &  X      
} \]
Since $u$ is flat, the base change homomorphism
$u^*{i_V}_*\Omega^p_{V/k}(\log\Sigma|_V)\to {{i_{V'}}_*}v^*\Omega^p_{V/k}(\log \Sigma|_V)$ 
is an isomorphism. Since $v$ is smooth, $v^*\Omega^p_{V/k}(\log\Sigma|_V) \to \Omega^p_{V'/k}(\log\Sigma'|_{V'})$ 
is an isomorphism. Therefore the base change isomorphism becomes 
$u^*\tilde{\Omega}^p_{X/k}(\log\Sigma)\isoto \tilde{\Omega}^p_{X'/k}(\log\Sigma')$.
\end{proof}

\begin{lem}\label{feBC}
Let $u \colon X'\to X$ be an \'etale morphism. Let $E'$ be a prime divisor on $X'$,
let $E$ be the closure of $u(E')$. Let $\omega$ be a rational $p$-form on $X$.
Then
\begin{itemize}
\item[i)] $\omega$ is regular at $E$ if and only if $u^*\omega$ is regular at $E'$.
\item[ii)] $\omega$ has at most a log pole at $E$ if and only if $u^*\omega$ has at most a 
log pole at $E'$.
\end{itemize}
\end{lem}

\begin{proof} i) There exists $l\gg 0$ such that $\omega\in \tilde{\Omega}^p_X(lE)_E$. 
By Lemma~\ref{eBC}, the inclusion $\tilde{\Omega}^p_{X/k}\subset \tilde{\Omega}^p_{X/k}(lE)$ becomes after \'etale 
pullback $\tilde{\Omega}^p_{X'/k}\subset \tilde{\Omega}^p_{X'/k}(lE')$. All sheaves that appear 
being coherent, the section
$\omega\in  \tilde{\Omega}^p_{X/k}(lE)_E$ belongs to $(\tilde{\Omega}^p_{X/k})_E$ if and only 
the pullback section $u^*\omega\in  \tilde{\Omega}^p_{X'/k}(lE')_{E'}$ belongs to $(\tilde{\Omega}^p_{X'/k})_{E'}$

ii) The proof in the logarithmic case is similar.
\end{proof}

\begin{lem}\label{eds} 
Let $u \colon X'\to X$ be an \'etale morphism. Let
$D$ be a $\Q$-Weil divisor on $X$. Let $D'=u^*D$ be the pullback $\Q$-Weil divisor, 
defined by restricting to big open subsets. Then 
$$
u^*\cO_X(\lfloor D\rfloor)\isoto \cO_{X'}(\lfloor D'\rfloor).
$$
\end{lem}

\begin{proof}
Restrict to the smooth loci of $X'$ and $X$. Round down commutes with pullback
since $u$ is unramified in codimension one. By flat base change, the isomorphism
extends to $X'$ and $X$.
\end{proof}


\subsection{Zariski-Steenbrink differentials and derivations on toric varieties~\cite{Dan78,Oda88}}

Let $M$ be a lattice, with dual lattice 
$N=M^*$. Let $T=T_N=\Hom(M,k^*)$ be the induced torus over $k$. The torus 
acts on the space of global regular functions, with eigenspace decomposition
$$
\Gamma(T,\cO_T)=\oplus_{m\in M} k\cdot \chi^m.
$$
Recall that each element $m\in M$ induces by evaluation a torus character $\chi^m\colon T\to k^*$.
Denote $\alpha_m=\frac{d(\chi^m)}{\chi^m}\in \Gamma(T,\Omega_{T/k})$.
The application $m\mapsto \alpha_m$ is additive, and induces an isomorphism
$$
\cO_T\otimes_\Z M\isoto \Omega_{T/k}, \ 1\otimes m\mapsto \alpha_m.
$$
This follows from computations on the affine space, since a choice of basis 
$m_1,\ldots,m_n$ of $M$, with induced characters $z_i=\chi^{m_i}$, identifies
$T$ with the principal open set $D(\prod_{i=1}^nz_i)\subset \bA^n_k$.
In particular, we obtain isomorphisms
$$
\cO_T\otimes_\Z \wedge^pM\isoto \Omega^p_{T/k}, \ 1\otimes m\mapsto \alpha_m.
$$
Passing to global sections, the image of $k\otimes_\Z M$ is $V\subset
\Gamma(T,\Omega_{T/k})$, the subspace of global $1$-forms invariant under the 
torus action. We have induced eigenspace decompositions
$$
\Gamma(T,\Omega^p_{T/k})=\oplus_{m\in M} \chi^m \cdot \wedge^p V.
$$
So every regular form $\omega\in \Gamma(T,\Omega^p_{T/k})$ admits a unique decomposition
$$
\omega=\sum_{m\in M}\chi^m\cdot \omega(m) \ (\omega(m)\in \wedge^pV).
$$

For $e\in N$, there exists a unique $k$-derivation $\theta_e$ of $\Gamma(T,\cO_T)$
such that $\theta_e(\chi^m)=\langle m,e\rangle \chi^m$ for every $m\in M$.
The application $\alpha\colon N\to \Gamma(T,\cT_{T/k}),e\mapsto \theta_e$ is additive
and induces an isomorphism of $\cO_T$-modules
$$
1\otimes_\Z \alpha \colon \cO_T\otimes_\Z N\to \cT_{T/k}. 
$$
Passing to global sections, the image of $k\otimes_\Z N$ is $W\subset \Gamma(T,\cT_{T/k})$,
the space of derivations invariant under the torus action. We have an eigenspace decomposition 
$$
\Gamma(T,\cT_{T/k})=\oplus_{m\in M}\chi^m W.
$$ 
So every $k$-derivation $\theta\colon k[M]\to k[M]$ has a unique decomposition
$$
\theta=\sum_{m\in M}\chi^m \theta(m)\ (\theta(m)\in W).
$$

As in~\cite[Lemma 4.3.1]{Dan78} or ~\cite[Proposition 3.1]{Oda88}, we obtain the following
explict formulas:

\begin{lem}\label{4.1} Let $e\in N$ be a primitive vector. Then 
$T_N\subset T_N\emb(\R_{\ge 0}e)=X$ is an affine torus embedding such that 
$X\setminus T$ consists of a unique prime divisor $E=V(e)$, and both $X/k$ and $E/k$ are smooth. 
We have eigenspace decompositions:
\begin{itemize}
\item[a)] $\Gamma(X,\Omega^p_{X/k})=\oplus_{\langle m,e\rangle=0}\chi^m\cdot 
\alpha(k\otimes_\Z \wedge^p(M\cap e^\perp)) \oplus \oplus_{\langle m,e\rangle>0}\chi^m \cdot\wedge^p V$.
\item[b)] $\Gamma(X,\tilde{\Omega}^p_{X/k}(\log E))=\oplus_{\langle m,e\rangle\ge 0}\chi^m\cdot \wedge^p V$.
\item[c)] $\Gamma(X,\cT_{X/k})=
\oplus_{\langle m,e\rangle=-1}\chi^m\cdot k\theta_e \oplus \oplus_{\langle m,e\rangle\ge 0}\chi^m \cdot W$.
\item[d)] $\Gamma(X,\tilde{\cT}_{X/k}(-\log E))= \oplus_{\langle m,e\rangle\ge 0} \chi^m \cdot W$.
\end{itemize}
\end{lem}

Let $\omega\in \Gamma(T,\Omega^p_{T/k})$, and denote
$
\Supp\omega=\{m\in M;\omega(m)\ne 0\}.
$
Let $\theta\in \Gamma(T,\cT_{T/k})$ and define similarly its support.
Let $T=T_N\subseteq X$ be a torus embedding, let $E=V(e)$ be an invariant prime divisor on $X$. 
It corresponds to a primitive vector $e\in N$. The following properties hold:
\begin{itemize}
\item[a)] $\omega$ is regular at $E$ if and only if for every $m\in \Supp \omega$, either
$\langle m,e\rangle>0$, or $\langle m,e\rangle=0$ and $\omega(m)\in \alpha(k\otimes_\Z \wedge^p(M\cap e^\perp))$.
\item[b)] $\omega$ has at most a logarithmic pole at $E$ if and only if $\langle m,e\rangle\ge 0$
for every $m\in \Supp\omega$.
\item[c)] $\theta$ is regular at $E$ if and only if for every $m\in \Supp \omega$, either
$\langle m,e\rangle\ge 0$, or $\langle m,e\rangle=-1$ and $\theta(m)\in k \theta_e$.
\item[d)] $\theta$ is regular logarithmic at $E$ if and only if $\langle m,e\rangle\ge 0$
for every $m\in \Supp\theta$.
\end{itemize}
Properties a)-d) do not depend on the toric model $X$, only on the valuation of $k(X)$
defined by $E$. So they follow from Lemma~\ref{4.1}.

\begin{thm}~\cite[Proposition 3.1]{Oda88}\label{logt}
Let $T_N\subseteq X$ be a torus embedding. The complement $\Sigma=X\setminus T$
is a reduced Weil divisor on $X$, and we have natural isomorphisms
$$
1\otimes_\Z\wedge^p\alpha\colon \cO_X\otimes_\Z \wedge^pM\to \tilde{\Omega}^p_X(\log\Sigma)
$$ 
$$
1\otimes_\Z \alpha\colon \cO_X\otimes_\Z N\to \tilde{\cT}_{X/k}(-\log\Sigma)
$$ 
\end{thm}

In particular, $\tilde{\Omega}^1_{X/k}(\log\Sigma)$ is a trivial $\cO_X$-module of rank
equal to the dimension of $X$, and 
$\wedge^p \tilde{\Omega}^1_{X/k}(\log\Sigma)\isoto \tilde{\Omega}^p_{X/k}(\log\Sigma)$.
And $\tilde{\cT}_{X/k}(-\log\Sigma)$ is a trivial $\cO_X$-module of rank
equal to the dimension of $X$.


\subsection{Toroidal embeddings~\cite{KKMS}}


A {\em toroidal embedding} is an open
subset $U\subseteq X$ in a normal variety $X/k$, such that for every $P\in X$,
there exists an affine toric variety $Z=T_N\emb(\sigma)$, a point $Q\in Z$, and 
an isomorphism of complete local $k$-algebras
$$
\hat{\cO}_{X,P}\simeq \hat{\cO}_{Z,Q}
$$
such that $X\setminus U$ corresponds to $Z\setminus T_N$. It follows that $U$
is nonsingular, and $\Sigma=X\setminus U$ has pure codimension one. If each
irreducible component of $\Sigma$ is normal, the toroidal embedding is called {\em strict}. 
A {\em toroidal morphism} $f\colon (U'\subseteq X')\to (U\subseteq X)$ of toroidal embeddings
is a morphism $f\colon X'\to X$, such that for every $P'\in X'$, we can choose local formal 
isomorphisms as above such that $\hat{\cO}_{X,f(P')}\to \hat{\cO}_{X',P'}$ corresponds to
the morphism $\hat{\cO}_{Z,g(Q')}\to \hat{\cO}_{Z',Q'}$ induced by a toric morphism $g\colon Z\to Z'$.
It follows that $f(U')\subseteq U$.

Given a local formal isomorphism as above, there exists by~\cite[Corollary 2.6]{Ar69} a hut
\[ 
\xymatrix{
   &    U'\subset X' \ni P' \ar[dl]_u  \ar[dr]^v &   \\
   U\subset X \ni P &  & T_N\subset Z\ni Q
} \]
with $u,v$ \'etale, $u(P')=P$, $v(P')=Q$, and $u^{-1}(U)=U'=v^{-1}(T_N)$.
By Theorem~\ref{logt} and Lemma~\ref{eBC}, we obtain 

\begin{thm} Let $U\subseteq X/k$ be a toroidal embedding. Then 
$\tilde{\Omega}^1_{X/k}(\log\Sigma)$ is a locally trivial $\cO_X$-module of rank equal to the dimension 
of $X$, and $\wedge^p \tilde{\Omega}^1_{X/k}(\log\Sigma)\isoto \tilde{\Omega}^p_{X/k}(\log\Sigma)$.
\end{thm}

\begin{prop}
Let $U\subseteq X$ and $V\subseteq Y$ be toroidal embeddings.
Let $f\colon X\to Y$ be a morphism such that $f(U)\subseteq V$. The pullback
homomorphism $\Omega_{V/k}^\bullet\to f_*\Omega^\bullet_{U/k}$ extends 
(uniquely) to a homomorphism
$$
\tilde{\Omega}_{Y/k}^\bullet(\log \Sigma_Y) \to f_*\tilde{\Omega}^\bullet_{X/k}(\log \Sigma_X).
$$
\end{prop}

\begin{proof} We prove the claim in two steps.

{\em Step 1:}  Suppose $T\subset Z$ is a torus embedding, $v\colon Y\to Z$ is a finite
\'etale morphism, and $V=v^{-1}(T)$. Let $M$ be the lattice of characters of the torus $T$. 
Let $m_1,\ldots,m_n$ be a basis of $M$, denote $t_i=\chi^{m_i}\ (1\le i\le n)$. Then 
$\tilde{\Omega}^1_{Z/k}(\log\Sigma_Z)$ is the free $\cO_Z$-module with basis 
$
\frac{dt_i}{t_i} \ (1\le i\le n),
$
and $\wedge^p \Omega^1_{Z/k}(\log\Sigma_Z)\isoto \Omega^p_{Y/k}(\log\Sigma_Y)$. 
By Lemma~\ref{eBC}, 
$v^*\tilde{\Omega}_{Z/k}^\bullet(\log \Sigma_Z)\isoto \tilde{\Omega}_{Y/k}^\bullet(\log \Sigma_Y)$.
Denote $z_i=v^*t_i\ (1\le i\le n)$. Then $z_i\in \Gamma(V,\cO_V^*)$, 
$\tilde{\Omega}^1_{Y/k}(\log\Sigma_Y)$ is the free $\cO_Y$-module with basis 
$$
\omega_i=\frac{dz_i}{z_i}\in \Gamma(Y,\tilde{\Omega}^1_{Y/k}(\log\Sigma_Y)) \ (1\le i\le n),
$$
and $\wedge^p \tilde{\Omega}^1_{Y/k}(\log\Sigma_Y)\isoto \tilde{\Omega}^p_{Y/k}(\log\Sigma_Y)$. 
Therefore, in order to prove the claim, it suffices to show that 
$f^*\omega_i \in \Gamma(X,\tilde{\Omega}^1_{X/k}(\log\Sigma_X))$.
But $g_i=f^*z_i\in \Gamma(U,\cO_U^*)$, and 
$$
f^*\omega_i=\frac{dg_i}{g_i}\in \Gamma(X,\tilde{\Omega}^1_{X/k}(\log\Sigma_X)).
$$
The claim holds in this case.
 
{\em Step 2:} The claim is local on $Y$, so we may shrink $Y$ to an affine open 
neighborhood of a fixed point. By~\cite{Ar69}, there exists a hut
\[ 
\xymatrix{
   &    V'\subset Y'  \ar[dl]_u  \ar[dr]^v &   \\
   V\subset Y  &  & T\subset Z
} \]
where $T\subseteq Z$ is a torus embedding, $u,v$ are \'etale, and $u^{-1}(V)=V'=v^{-1}(T)$.
Denote $X'=X\times_Y Y'$, and consider the base change diagram
\[ 
\xymatrix{
 X \ar[d]_f  & X'  \ar[d]^{f'}  \ar[l]_{u'} \\
 Y              &  Y'     \ar[l]^u
} \]
Denote $U'={u'}^{-1}(U)$. The restriction of the cartesian diagram to open subsets is also 
cartesian
\[ 
\xymatrix{
 U \ar[d]_f  & U'  \ar[d]^{f'}  \ar[l]_{u'} \\
 V              &  V'     \ar[l]^u
} \]
Since $u'$ is \'etale, $U'\subseteq X'$ is also a toroidal embedding. 
Let $\omega\in \Gamma(Y,\Omega^p_{Y/k}(\log \Sigma_Y))$. Then 
$\omega|_V\in  \Gamma(V,\Omega^p_{V/k})$, so $\eta=f^*(\omega|_V)\in \Gamma(U,\Omega^p_{U/k})$.
We have ${u'}^*\eta={f'}^*u^*\omega$. By Lemma~\ref{eBC}, $u^*\omega\in 
\Gamma(Y',\tilde{\Omega}^p_{Y'/k}(\log \Sigma_{Y'}))$. By Step 1, ${f'}^*u^*\omega\in
\Gamma(X',\tilde{\Omega}^p_{X'/k}(\log \Sigma_{X'}))$. Therefore ${u'}^*\eta$ has at most 
logarithmic poles along the primes of $\Sigma_{X'}$. By Lemma~\ref{feBC}, 
$\eta\in \Gamma(X,\tilde{\Omega}^p_{X/k}(\log \Sigma_X))$.
\end{proof}


\subsection{Log smooth embeddings}

A {\em log smooth embedding} is a toroidal embedding $U\subseteq X$ 
such that $X/k$ is smooth. If we denote $\Sigma=X\setminus U$, this is equivalent
to $(X,\Sigma)$ being a {\em log smooth pair}, that is $X/k$ is smooth and the restriction 
of $\Sigma$ to each connected component of $X$ is either empty, or a normal crossing
divisor. A log smooth embedding is called {\em strict} if it is so as a toroidal embedding.
This is equivalent to the property that each irreducible component of $\Sigma$ is
smooth, that is $\Sigma$ is a simple normal crossing divisor.
We obtain an equivalence between (strict) log smooth embeddings and (strict) log smooth pairs.

We assume $\Char k=0$. We need the special case $E=0$ of~\cite[Theorem 3.4]{BM11A}:

\begin{thm}[Bierstone-Milman]\label{sr} 
Let $X$ be a smooth irreducible variety, let $\Sigma$ be a reduced divisor on $X$.
Let $V$ be an open subset of $X$ such that $\Sigma|_V$ has at most simple normal crossing
singularities. Then there exists a proper morphism $\sigma\colon X'\to X$ such that
\begin{itemize}
\item[a)] $X'$ is smooth and $\sigma^{-1}\Sigma$ has at most simple normal crossing singularities.
\item[b)] $\sigma\colon \sigma^{-1}(V)\to V$ is an isomorphism.
\end{itemize}
\end{thm}

\begin{lem}\label{hr}
Let $U\subseteq X$ be an open dense embedding, with $U$ smooth. Then 
there exists a proper morphism $\sigma\colon X'\to X$ such that 
$\sigma\colon \sigma^{-1}(U)\to U$ is an isomorphism and $\sigma^{-1}(U)\subseteq X'$
is a strict log smooth open embedding.
\end{lem}

\begin{proof} Since $U$ is smooth, the singular locus of $X$ is contained in $X\setminus U$. 
By Hironaka's strong resolution of singularities, there exists a
proper morphism $\sigma\colon X'\to X$ such that $X'$ is smooth, $\sigma^{-1}(U)\to U$
is an isomorphism, and the complement of $\sigma^{-1}(U)$ in $X$ is a SNC divisor.
The open embedding $\sigma^{-1}(U)\subseteq X'$ is therefore strict log smooth, and 
satisfies the desired properties.
\end{proof}

\begin{lem}\label{2.3} Let $U\subseteq X$ and $X\subseteq Y$ be open dense embeddings.
\begin{itemize}
\item[1)] If $U\subseteq Y$ is a strict log smooth embedding, so is $U\subseteq X$.
\item[2)] Suppose $U\subseteq X$ is a strict log smooth embedding. Then there exists 
a proper morphism $\sigma\colon Y'\to Y$ such that $\sigma\colon \sigma^{-1}(X)\to X$ is an 
isomorphism and the open embedding $\sigma^{-1}(U)\subseteq Y'$ is strict log smooth.
\end{itemize}
\end{lem}

\begin{proof} 1) Since $Y$ is smooth, so is $X$. The divisor $Y\setminus U$ is SNC on $Y$,
hence so is its restriction $(Y\setminus U)|_X=X\setminus U$. Therefore $U\subseteq X$ is 
strict log smooth.

2) Since $U\subseteq X$ is strict log smooth, $X$ is smooth. By Lemma~\ref{hr} for $X\subseteq Y$, 
we may replace $Y$ by a modification outside $X$, so that $X\subseteq Y$ is also strict log smooth.
Let $\Sigma=Y\setminus U=(Y\setminus X)\cup (X\setminus U)$. Then $\Sigma$ is a divisor on $Y$.
Its restriction to $X$ is $\Sigma|_X=X\setminus U$, a SNC divisor by assumption. By Theorem~\ref{sr}, 
we may replace $Y$ by a modification outside $X$ so that $\Sigma$ becomes a SNC divisor on $Y$. 
Therefore $U\subseteq Y$ is strict log smooth.
\end{proof}

\begin{cor}\label{SNCemb} Let $U\subseteq X$ be a strict log smooth embedding.
Then there exists an open embedding $j\colon X\subseteq \bar{X}$ such that the induced open
embedding $U\subseteq \bar{X}$ is strict log smooth, and $\bar{X}/k$ is proper. 
\end{cor}

\begin{proof} By Nagata, there exists an open dense embedding $X\subseteq \bar{X}$, 
with $\bar{X}/k$ proper. By Lemma~\ref{2.3}.2), we may replace $\bar{X}$ by a modification
outside $X$, so that the induced embedding $U\subseteq \bar{X}$ is strict log smooth.
\end{proof}

\begin{cor}\label{moremb}
Let $U\subseteq X$ and $V\subseteq Y$ be strict log smooth open embeddings.
Let $f\colon X\to Y$ be a morphism such that $f(U)\subseteq V$. Then there exists 
a commutative diagram
\[ 
\xymatrix{
 U \ar[d]  \ar[r] & X  \ar[d]^f \ar[r] & \bar{X}  \ar[d]^{\bar{f}}  \\
  V    \ar[r]                   &  Y   \ar[r]         & \bar{Y}  
} \]
such that
\begin{itemize}
\item[a)]  the vertical arrows are open embeddings. 
\item[b)] $U\subseteq \bar{X}$ and $V\subseteq \bar{Y}$ are strict log smooth embeddings.
\item[c)] $\bar{X}$ and $\bar{Y}$ are proper over $k$.
\end{itemize}
Moreover, $f$ is proper if and only if $X={\bar{f}}^{-1}(Y)$.
\end{cor}

\begin{proof} By Corollary~\ref{SNCemb}, there exists an open embedding $Y\subseteq \bar{Y}$
such that $\bar{Y}$ is proper and $V\subseteq \bar{Y}$ is strict log smooth. 
By Nagata, there exists an open embedding $X\subseteq X'$ with $X'$ proper. 
Then $f$ induces a rational map $\bar{f}\colon X'\dashrightarrow \bar{Y}$.
Let $\Gamma$ be the graph of $\bar{f}$, with induced morphisms to $X'$ and $\bar{Y}$,
which partially resolve $\bar{f}$. Since $\bar{f}$ is defined over $X$, $\Gamma\to X'$ is an isomorphism 
over $X$. We obtain a chain of open embeddings $U\subseteq X\subseteq \Gamma$. By 
Lemma~\ref{2.3}.2), we may replace $\Gamma$ by a modification outside $X$, denoted $\bar{X}$, such 
that $U\subseteq\bar{X}$ is strict log smooth. This ends the construction of the diagram.
The last statement follows from Lemma~\ref{pc}.
\end{proof}

\begin{lem}\label{pc}
Let $f\colon X\to S$ be a proper morphism of schemes.
Let $U\subseteq X$ and $V\subseteq S$ be open dense subsets such that $f(U)\subseteq V$. 
The following properties are equivalent:
\begin{itemize}
\item[1)] The induced morphism $g\colon U\to V$ is proper.
\item[2)] $U=f^{-1}(V)$.
\end{itemize}
\end{lem}

\begin{proof} $1)\Longrightarrow 2)$: Consider the commutative diagram
\[ 
\xymatrix{
U \ar[dr]_g\ar[rr]^\iota    & & f^{-1}(V) \ar[dl]^{f|_{f^{-1}(V)}}  \\ 
 & V &
} \]
Since $f|_{f^{-1}(V)}$ is proper and $g$ is separated, it follows
that the open embedding $\iota$ is proper. Since $U$ is also
dense in $f^{-1}(V)$, we obtain $U=f^{-1}(V)$.

$2)\Longrightarrow 1)$: The morphism $g$ is obtained 
from $f$ by base change with open embedding $V\subset S$. Therefore it
is proper.
\end{proof}


\subsection{Hypercohomology with supports}


Let $X$ be an algebraic variety. Let $U\subseteq X$ be an open subset, let $Z=X\setminus U$.
Let $\alpha\colon \cA\to \cB$ be a homomorphism of bounded below complexes of $\cO_X$-modules. 

\begin{lem}\label{eqi}
Suppose the natural maps in hypercohomology induced by $\alpha$ and $\alpha|_U$ 
$$
\bH^*(X,\cA)\to \bH^*(X,\cB), \ \bH^*(U,\cA|_U)\to \bH^*(U,\cB|_U)
$$
are isomorphisms. Then the natural map induced by $\alpha$ in hypercohomology with support in $Z$
is also an isomorphism:
$$
\bH_Z^*(X,\cA)\isoto \bH_Z^*(X,\cB).
$$
\end{lem}

\begin{proof} The long exact sequences for hypercohomology with supports induce a commutative diagram
with exact rows
\[ 
\xymatrix{
\bH^{i-1}(X,\cA) \ar[d] \ar[r] &  \bH^{i-1}(U,\cA) \ar[d]  \ar[r] & \bH^i_Z(X,\cA) \ar[d] \ar[r] & \bH^i(X,\cA) \ar[r] \ar[d] &  H^i(U,\cA) \ar[d] \\
\bH^{i-1}(X,\cB) \ar[r] & \bH^{i-1}(U,\cB) \ar[r] & \bH^i_Z(X,\cB)  \ar[r] & \bH^i(X,\cB)  \ar[r] & \bH^i(U,\cB)
} \]
All but the middle vertical arrows are isomorphisms. Then the middle vertical arrow is also an isomorphism,
by the 5-lemma.
\end{proof}

\begin{lem}\label{d0} 
Suppose $\alpha$ is a quasi-isomorphism over $U$. If $\dim Z=0$ and the natural maps
$$
\bH_Z^*(X,\cA)\to \bH_Z^*(X,\cB)
$$
are isomorphisms, then $\alpha$ is a quasi-isomorphism.
\end{lem}

\begin{proof} This follows from the local to global spectral sequence, cf.~\cite[page 196]{PS}.
\end{proof}

\begin{lem}
Suppose $Z$ is a finite disjoint union of closed subsets $Z_i$. Then the natural homomorphisms
$
\oplus_i \bH^*_{Z_i}(X,\cA)\to \bH^*_Z(X,\cA)
$
are isomorphisms.
\end{lem}

\begin{proof}
By induction on the cardinality of the $Z_i$'s, and the Mayer-Vietoris sequence.
\end{proof}


\subsection{Invariance of logarithmic sheaves}


Suppose $\Char k=0$. To avoid heavy notation, we denote 
$\tilde{\Omega}^p_{X/k}(\log\Sigma)$ by $\Omega^p_X(\log\Sigma)$.

\begin{thm}\label{sc}
Let $(X',\Sigma'),(X,\Sigma)$ be strict log smooth pairs, let $f\colon X'\to X$ be a 
proper morphism such that $f\colon X'\setminus \Sigma'\to X\setminus \Sigma$
is an isomorphism. Then for every $p$, the natural homomorphism 
$$
\Omega^p_X(\log\Sigma)\to Rf_*\Omega^p_{X'}(\log\Sigma')
$$
is a quasi-isomorphism.
\end{thm}

\begin{proof} Denote $\alpha\colon \Omega^p_X(\log\Sigma)\to Rf_*\Omega^p_{X'}(\log\Sigma')$.
We prove by induction on $\dim X$ that $\alpha$ is a quasi-isomorphism. If $\dim X=1$, then 
$f$ is an isomorphism and the claim is clear. Suppose $\dim X\ge 2$. Let $Z$ be the complement
of the largest open subset of $X$ where $\alpha$ is a quasi-isomorphism. It is the union of the 
supports of the following $\cO_X$-modules: the cokernel $\cC$ of 
$\Omega^p_X(\log\Sigma)\to f_*\Omega^p_{X'}(\log\Sigma')$, and 
$R^if_*\Omega^p_{X'}(\log\Sigma')\ (i>0)$. Suppose by contradiction that $Z$ is nonempty.

{\em Step 1:} We claim that $\dim Z\le 0$. Indeed, the statement is local on $X$, so
we may suppose $X$ is affine. Let $H$ be a general hyperplane section of $X$.
Denote $H'=f^*H$. Then $(H,\Sigma|_H)$ and $(H',\Sigma'|_{H'})$ are strict log smooth,
and $g=f|_{H'} \colon H'\to H$ maps $H\setminus (\Sigma|_H)$ isomorphically onto 
$H'\setminus (\Sigma'|_{H'})$. By induction, 
$
\Omega^p_H(\log\Sigma|_H)\to Rg_*\Omega^p_{H'}(\log\Sigma'|_{H'})
$
is a quasi-isomorphism. 

Since $H$ is general, we have base change isomorphisms
$$
R^if_*\Omega^p_{X'}(\log\Sigma')|_H\isoto R^ig_*\Omega^p_{H'}(\log\Sigma'|_{H'}) \ (i\ge 0).
$$
For $i>0$, the right hand side is zero, and therefore $R^if_*\Omega^p_{X'}(\log\Sigma')|_H$ is zero.
For $i=0$, consider the commutative diagram with exact raws 
\[ 
\xymatrix{
 \Omega^p_X(\log \Sigma)|_H \ar[d]^r  \ar[r] & f_*\Omega^p_{X'}(\log\Sigma')|_H  \ar[d]^{r'} \ar[r] & \cC|_H \ar[d]^{r''} \ar[r] & 0 \\
  \Omega^p_H(\log \Sigma|_H)    \ar[r]^\simeq          &  g_*\Omega^p_{H'}(\log\Sigma'|_{H'})  \ar[r]         & 0   & 
} \]
where $r''$ is induced by $r$ and $r'$. Since $r'$ is injective and $r$ is surjective, it follows that 
$r''$ is injective. Therefore $\cC|_H=0$.

We conclude that $Z\cap H=\emptyset$. Therefore $\dim Z\le 0$.

{\em Step 2:} We claim that $\bH^*_Z(X,\alpha)$ is an isomorphism. Indeed,
denote $U'=X'\setminus \Sigma'$, $U=X\setminus \Sigma$. Then $U'\subseteq X'$ and $U\subseteq X$ are strict
log smooth embeddings, and $f\colon X'\to X$ maps $U'$ isomorphically onto $U$.
We compactify this data as in Corollary~\ref{moremb}:
\[ 
\xymatrix{
 U' \ar[d]  \ar[r] & X'  \ar[d]^f \ar[r] & \bar{X'}  \ar[d]^{\bar{f}}  \\
 U    \ar[r]          & X   \ar[r]         & \bar{X}  
} \]
Since $f$ is proper, it is obtained from $\bar{f}$ by base change with the open embedding
$X\subseteq \bar{X}$. Denote $\bar{\Sigma'}=\bar{X'}\setminus U'$, 
$\bar{\Sigma}=\bar{X}\setminus U$. Consider the natural homomomorphism
$$
\bar{\alpha}\colon \Omega^p_{\bar{X}}(\log\bar{\Sigma})\to R\bar{f}_*\Omega^p_{\bar{X'}}(\log\bar{\Sigma'}).
$$
Since the second square is cartesian, $\bar{\Sigma'}|_{X'}=\Sigma'$, $\bar{\Sigma}|_X=\Sigma$, and 
logarithmic sheaves commute with base change by open embeddings, we obtain an identification
$$
\bar{\alpha}|_X\isoto \alpha.
$$
Denote $\bar{Z}=Z\cup (\bar{X}\setminus X)$ and $\bar{U}=\bar{X}\setminus \bar{Z}=X\setminus Z$.
From the quasi-isomorphism $\bar{\alpha}|_{\bar{U}}\isoto \alpha|_{X\setminus Z}$, we deduce that $\bar{\alpha}$ 
is a quasi-isomorphism over $\bar{U}$. Therefore we obtain isomorphisms
$$
\bH^*(\bar{U},\Omega^p_{\bar{X}}(\log\bar{\Sigma})|_{\bar{U}})\to 
\bH^*(\bar{U},R\bar{f}_*\Omega^p_{\bar{X'}}(\log\bar{\Sigma'})|_{\bar{U}}).
$$

Next, we claim that the homomorphism
$
\bH^*(\bar{X},\Omega^p_{\bar{X}}(\log\bar{\Sigma}))\to 
\bH^*(\bar{X},R\bar{f}_*\Omega^p_{\bar{X'}}(\log\bar{\Sigma'}))
$
is also an isomorphism. Indeed, it identifies with the homomorphism
$$
H^*(\bar{X},\Omega^p_{\bar{X}}(\log\bar{\Sigma}))\to H^*(\bar{X'},\Omega^p_{\bar{X'}}(\log\bar{\Sigma'})).
$$
This is an isomorphism by the Atiyah-Hodge lemma and Deligne's theorem on E$_1$ degeneration for 
logarithmic de Rham complexes, since $\bar{X'},\bar{X}$ are proper and $\bar{f}\colon \bar{X'}\setminus \bar{\Sigma'}
\to \bar{X}\setminus \bar{\Sigma}$ is the isomorphism $f\colon U'\isoto U$.

By Lemma~\ref{eqi}, $\bH^*_{\bar{Z}}(\bar{X},\bar{\alpha})$ is an
isomorphism. Since $\bar{Z}$ is the disjoint union of $Z$ with $\bar{X}\setminus X$, we deduce
that $\bH^*_Z(\bar{X},\bar{\alpha})$ is an isomorphism. Since $Z$ is contained in the open subset $X$
of $\bar{X}$ and $\bar{\alpha}|_X=\alpha$, it follows by excision that $\bH^*_Z(X,\alpha)$ is an isomorphism.

{\em Step 3:} Since $\dim Z=0$ and $\bH^*_Z(X,\alpha)$ is an isomorphism, 
Lemma~\ref{d0} implies that $\alpha$ is a quasi-isomorphism. That is $Z=\emptyset$, a contradiction.
\end{proof}

\begin{cor}\label{fst}
Let $(X',\Sigma'),(X,\Sigma)$ be toroidal pairs, let $f\colon X'\to X$ be a 
proper morphism such that $f\colon X'\setminus \Sigma'\to X\setminus \Sigma$
is an isomorphism. Then for every $p$, the natural homomorphism 
$$
\Omega^p_X(\log\Sigma)\to Rf_*\Omega^p_{X'}(\log\Sigma')
$$
is a quasi-isomorphism.
\end{cor}

\begin{proof} We prove the claim in several steps.

{\em Step 0:} If moreover $(X',\Sigma')$ is strict log smooth, if suffices to check the
claim for a particular $f$. Indeed, suppose $g\colon (X'',\Sigma'')\to (X,\Sigma)$
is another morphism with the same properties, with $(X'',\Sigma'')$ strict log smooth, and
we know the claim holds for $g$. There exists a Hironaka hut
\[ 
\xymatrix{
   &    (X''',\Sigma''') \ar[dl] \ar[dr] &   \\
   (X',\Sigma') \ar[dr]  &  & (X'',\Sigma'') \ar[dl] \\
   &  (X,\Sigma) &
} \]
such that $(X''',\Sigma''')$ is strict log smooth, and all arrows are isomorphisms
above $X\setminus \Sigma$. The claim holds for $X''/X$ by assumption, and for $X'''/X''$ by 
Theorem~\ref{sc}. Therefore it holds for $X'''/X$. By Theorem~\ref{sc}, it also holds for $X'''/X'$.
Therefore it holds for $X'/X$.

{\em Step 1:} Suppose $(X',\Sigma'),(X,\Sigma)$ and $f$ are toric. In this case,
$\Omega^p_X(\log\Sigma)\simeq \cO_X^{\oplus r}$ for some $r$, and 
$f^*\Omega^p_X(\log\Sigma)\to \Omega^p_{X'}(\log\Sigma')$ is an isomorphism.
By the projection formula, our homomorphism is a quasi-isomorphism if and only if 
$$
\cO_X\to Rf_*\cO_{X'}
$$
is a quasi-isomorphism. This holds, and can be proved combinatorially~\cite{Dan78}.

{\em Step 2:} Suppose $(X',\Sigma')$ is strict log smooth and $(X,\Sigma)$ is toric.
There exists a toric log resolution $(X'',\Sigma'')\to (X,\Sigma)$, which by 
construction is an isomorphism over the torus $T=X\setminus \Sigma$. Moreover,
$(X'',\Sigma'')$ is strict log smooth. By Step 1, the claim holds for $X''/X$.
By Step 0, it also holds for $X'/X$.

{\em Step 3:} $(X',\Sigma')$ is strict log smooth and $(X,\Sigma)$ is toroidal.
The claim is local on $X$. After possibly shrinking $X$ near a fixed point, there exists 
a hut
\[ 
\xymatrix{
   &    (Y,\Sigma_Y) \ar[dl] \ar[dr] &   \\
   (X,\Sigma)  &  & (Z,\Sigma_Z) 
} \]
such that $Y/X$ and $Y/Z$ are \'etale, $\Sigma_Y$ is the preimage of both $\Sigma$ and 
$\Sigma_Z$, and $(Z,\Sigma_Z)$ is the restriction to an open subset of a toric pair.
Our sheaves commute with \'etale base change, so our claim on $X$ is equivalent to
the claim for the pullback of $(X',\Sigma')\to (X,\Sigma)$ to $Y$.
On the other hand, we can construct a toric resolution, which when restricted to the open 
subset $Z$ will satisfy the claim. After base change to $Y$, the claim still holds.
We obtain two morphisms $(Y',\Sigma')\rightarrow (Y,\Sigma_Y)\leftarrow (Y'',\Sigma'')$
as in the claim, with $(Y',\Sigma'),(Y'',\Sigma'')$ strict log smooth. The claim holds for $Y''/Y$.
By Step 0, it also holds for $Y'/Y$. Therefore it holds for $X'/X$.

{\em Step 4:} By Hironaka, there exists a diagram
\[ 
\xymatrix{
   (X',\Sigma') \ar[d]  & (X'',\Sigma'') \ar[dl] \ar[l]   \\
   (X,\Sigma)  &  
} \]
such that $(X'',\Sigma'')$ is strict log smooth and $X''/X'$ is an isomorphism over $X'\setminus \Sigma'$.
By Step 3, the claim holds for $X''/X$ and $X''/X'$. Therefore it holds for $X'/X$.
\end{proof}


\section{Roots of sections}


Let $X$ be a scheme and $\cL$ an invertible $\cO_X$-module. For $n\in \Z$,
denote the tensor product $\cL^{\otimes n}$ by $\cL^n$. 

\begin{prop} Consider a global section $s\in \Gamma(X,\cL^n)$, for some positive integer $n$. 
Then there exist a morphism of schemes $\pi\colon Y\to X$ and a global section $t\in \Gamma(Y,\pi^*\cL)$, 
such that $t^n=\pi^*s$, and the following universal property holds: 
if $g\colon Y'\to X$ is a morphism of schemes, and $s'\in \Gamma(Y',g^*\cL)$ 
is a global section such that ${s'}^n=g^*s$, then there exists a unique morphism 
$u\colon Y'\to Y$ such that $g=\pi\circ u$ and $s'=u^*t$.
\end{prop}

\begin{proof}
{\em Step 1:} Suppose $X=\Spec A$ and $\cL=\cO_X$. Then $s\in \Gamma(X,\cO_X)=A$.
The ring homomorphism
$$
A\to \frac{A[T]}{(T^n-s)}
$$
induces a finite morphism $\pi\colon Y\to X$. If we denote by $t\in \Gamma(Y,\cO_Y)$ the class 
of $T$, we have $t^n=\pi^*s$. Let $g\colon Y'\to X$ be a morphism of schemes, and 
$s'\in \Gamma(Y',\cO_{Y'})$ with ${s'}^n=g^*s$. There exists a unique homomorphism of $A$-algebras
$$
\frac{A[T]}{(T^n-s)}\to \Gamma(Y',\cO_{Y'})
$$
which maps $T$ to $s'$. This translates into a morphism $u\colon Y'\to Y$ with $g=\pi\circ u$ and 
$s'=u^*t$.

{\em Step 2:} Consider the $\cO_X$-algebra  
$
\cA=\oplus_{i=0}^{n-1}\cL^{-i},
$
with the following multiplication: if $u_i,u_j$ are local sections of $\cA_i$ and $\cA_j$ respectively,
their product is the local section $u_i\otimes u_j$ of $\cA_{i+j}$ if $i+j<n$, and the local section
$u_i\otimes u_j\otimes s$ of $\cA_{i+j-n}$ if $i+j\ge n$. Let 
$\pi\colon Y=\Spec_X(\cA)\to X$ be the induced finite morphism of schemes.

Let $u\in \Gamma(U,\cL)$ be a nowhere zero section on some open subset $U\subseteq X$.
Then $s|_U=fu^n$ for some $f\in \Gamma(U,\cO_U)$.
Then $\cL|_U=\cO_U u$ and $\cA|_U=\oplus_{i=0}^{n-1}\cO_U u^{-i}$, $u^{-1}\in \Gamma(U,\cA_1)$
satisfies $(u^{-1})^n=f$, and mapping $T\mapsto u^{-1}$ induces an isomorphism over $U$
$$
\pi^{-1}(U)\isoto \Spec_U \frac{\cO_U[T]}{(T^n-f)}.
$$
Therefore the construction of $\cA$ globalizes the local construction in Step 1.

Consider the section $u^{-1}\cdot \pi^*u\in \Gamma(\pi^{-1}(U),\pi^*\cL)$. It satisfies 
$(u^{-1}\cdot \pi^*u)^n=\pi^*(s|_U)$. If $u'$ is another nowhere zero global section of $\cL|_U$, 
then $u'=vu$ for some unit $v\in \Gamma(U,\cO_U^\times)$. Since $\pi^*v=v$, we obtain 
${u'}^{-1}\cdot \pi^*u'=u^{-1}\cdot \pi^*u$. So the section does not depend on the choice of $u$.
Since $X$ can be covered by affine open subsets which trivialize $\cL$, it follows
that $u^{-1}\cdot \pi^*u$ glue to a section $t$ of $\pi^*\cL$ whose $n$-th power is $\pi^*s$.
The universal property can be checked on affine open subsets of $X$ on which $\cL$
is trivial, so it follows from Step 1.
\end{proof}

The morphism $\pi\colon Y\to X$, endowed with the section $t\in \Gamma(Y,\pi^*\cL)$, 
is unique up to an isomorphism over $X$. It is called the {\em $n$-th root of $s$}. 
We denote $Y$ by $X[\sqrt[n]{s}]$, and $t$ by $\sqrt[n]{s}$. 

\begin{lem}\label{bp}
Let $s\in \Gamma(X,\cL^n)$ be a global section for some $n\ge 1$, let 
$\pi\colon X[\sqrt[n]{s}]\to X$ be the $n$-th root of $s$. The following properties hold:
\begin{itemize}
\item[a)] $\pi_*\cO_{X[\sqrt[n]{s}]}=\oplus_{i=0}^{n-1}\cL^{-i}$.
\item[b)] Suppose the group of units $\Gamma(X,\cO_X^*)$ contains a primitive $n$-th
root of $1$. Then $\Z/n\Z$ acts on $X[\sqrt[n]{s}]$, the morphism $\pi$ is the induced quotient map, 
and a) is the decomposition into eigensheaves.
\item[c)] The morphism $\pi$ is finite and flat. 
\item[d)] Let $f\colon X'\to X$ be a morphism. Then the $n$-th root of 
$f^*s\in \Gamma(X',f^*\cL^n)$ is the pullback morphism 
$X[\sqrt[n]{s}]\times_X X'\to X'$, endowed with the pullback section.
\end{itemize}
\end{lem}

\begin{proof} All properties follow from the local description of the $n$-th root in the case 
when $X$ is affine and $\cL=\cO_X$. Note that $\pi$ is flat since $\cL$ is locally free.
In b), let $\zeta\in \Gamma(X,\cO_X^*)$ be a primitive $n$-th root of $1$. The action is 
$(\zeta,aT^i)\mapsto \zeta^i a T^i$ for the local model in Step 1, 
and $(\zeta,u_i)\mapsto \zeta^iu_i$ for the global model in Step 2.
\end{proof}

\begin{exmp}
Consider $\bA^1_\Z=\Spec \Z[T]$, and $n\ge 1$. View $T$ as a global section of 
$\cO_{\bA^1_\Z}^n=\cO_{\bA^1_\Z}$. The $n$-th root of $T$ is the endomorphism
$\pi_n\colon \bA^1_\Z\to \bA^1_\Z$ induced by $T\mapsto T^n$, and the global section is again $T$.
Indeed, $(\pi_n,T)$ satisfies the universal property. Since roots commute with base change
by Lemma~\ref{bp}.d), we also obtain the following description:
let $f\in \Gamma(X,\cO_X)$ and $n\ge 1$.
View $f$ as a global section of $\cO_X^n=\cO_X$. The $n$-th root of $f$
coincides with $X\times_{\bA^1_\Z}\bA^1_\Z$, where $f\colon X\to \bA^1_\Z$ is the morphism
induced by $f$, and $\pi_n \colon \bA^1_\Z\to \bA^1_\Z$ is the $n$-th root of $T$. 
We obtain a cartesian diagram
$$
\xymatrix{
X   \ar[d]_f  & X[\sqrt[n]{f}]  \ar[l]_\pi \ar[d] \\
\bA^1_\Z    &  \bA^1_\Z \ar[l]_{\pi_n}
}
$$ 
\end{exmp}

\begin{rem} Let $m,n\ge 1$ and $s\in \Gamma(X,\cL^{mn})$. Then the $mn$-th root of $s$ is the 
composition of two roots: the $m$-th root of $s\in \Gamma(X,(\cL^n)^m)$, followed by the 
$n$-th root of $\sqrt[m]{s}$. Indeed, this composition satisfies the universal property.
\end{rem}

\begin{exmp}
If $s$ vanishes nowhere and $n\in \Gamma(X,\cO^*_X)$, 
then $X[\sqrt[n]{s}]\to X$ is a finite \'etale covering. 
\end{exmp}

\begin{rem}\label{unc}
Let $A$ be a ring, $a\in A$ and $u\in U(A)$. Then $A[\sqrt[n]{a}]\isoto A[\sqrt[n]{u^na}]$ over $A$.
Indeed, $T\mapsto uT$ induces an $A$-isomorphism
$
A[T]/(T^n-a)\isoto A[T]/(T^n-u^na).
$
\end{rem}

\begin{rem} Suppose $X$ is reduced and irreducible, and $n\in \Gamma(X,\cO^*_X)$. Let $s_1,s_2$ be two 
nonzero global sections of $\cL^n$ with the same zero locus, an effective Cartier divisor $D$.
Then $s_2=us_1$ for some unit $u\in \Gamma(X,\cO_X^\times)$. The two cyclic covers 
$X[\sqrt[n]{s_i}] \to X \ (i=1,2)$ become isomorphic 
after base change with the \'etale cover $\tau\colon X[\sqrt[n]{u}]\to X$.
If $X/k$ is proper, then $\Gamma(X,\cO_X)=k$, so $u\in k^\times$. Therefore $\sqrt[n]{u}\in k^\times$
so $\tau$ is an isomorphism. It follows that the two cyclic covers are already isomorphic over $X$. 
Therefore, if $X/k$ is integral and proper, we can speak of the {\em cyclic cover 
associated to $\cO_X(D)\simeq \cL^n$}.
\end{rem}

Even if $X$ is reduced and $s$ is nowhere zero, the scheme $X[\sqrt[n]{s}]$ may not be reduced. 
For example, $\bF_2[T]/(T^2-1)$ has nilpotents, as $T^2-1=(T-1)^2$.

\begin{rem} Suppose $\zeta\in \Gamma(X,\cO_X^\times)$ satisfies $\zeta^q=1$, and $q\ge 1$ is minimal
with this property. Let $s\in \Gamma(X,\cL^n)$ and $n,q\ge 1$. Then $T^{nq}-s^q=\prod_{\zeta\in \mu_q}(T^n-\zeta s)$.
Therefore $X[\sqrt[nq]{s^q}]=\cup_{\zeta\in \mu_q}X[\sqrt[n]{\zeta s}]$.
\end{rem}

\begin{exmp}[Singularities of semistable reduction]
Let $X=\bA^d_k$ and $s=\prod_{i=1}^d z_i^{m_i}\in \Gamma(X,\cO_X)$, for some 
$(m_1,\ldots,m_d)\in \N^d\setminus 0$. Let $n\ge 2$ with $\Char k\nmid n$. 
The $n$-th root of $s$ is the hypersurface 
$$
X[\sqrt[n]{s}]=Z(t^n-\prod_{i=1}^d z_i^{m_i})\subset \bA^{d+1}_k.
$$
It is smooth if and only if $d=1$ and $m_1=1$. Else, its singular locus is 
$$
\cup_{m_i\ge 2}Z(t,z_i)\cup \cup_{m_i=m_j=1,i\ne j}Z(t,z_i,z_j).
$$
The components of the former (latter) kind have codimension one (two) in $X[\sqrt[n]{s}]$.
Therefore $X[\sqrt[n]{s}]$ is normal if and only if $\max_i m_i=1$. 

Denote $g=\gcd(n,m_1,\ldots,m_d)$. Let $n=gn',m_i=gm'_i$. Denote $s'=\prod_{i=1}^d z_i^{m'_i}$. 
Then $X[\sqrt[n]{s}]$ is a reduced $k$-variety, with irreducible decomposition
$$
X[\sqrt[n]{s}] = \cup_{\zeta\in \mu_g} X[\sqrt[n']{\zeta s'}].
$$

Each irreducible component is isomorphic over $X$ to $X[\sqrt[n']{s'}]$.
The latter is the simplicial toric variety 
$T_\Lambda\emb(\sigma)$, where $\Lambda=\{\lambda\in \Z^d; \sum_{i=1}^d\lambda_i m_i\in r\Z\}$ 
and $\sigma\subset \R^d$ is the standard positive cone (cf.~\cite[page 98]{KKMS}, 
\cite[Lemma 2.2]{Ste76}, \cite[Example 9.9, Proposition 10.10]{Kol95}).

If $d=1$, the root is easier to describe. The normalization $\bar{X}[\sqrt[n']{s'}] \to X[\sqrt[n']{s'}]$ is 
$\bA^1\to \bA^2, \ x \mapsto (x^{m_1},x^{n'})$. So $X[\sqrt[n]{s}]$ consists of $g$ lines
through the origin in the affine plane, each line being isomorphism 
over $X$ to the morphism $\bA^1\to \bA^1, x\mapsto x^{n'}$. 
\end{exmp}


\subsection{Roots of torus characters}


Let $T_N$ be a torus defined over a field $k$. Let $M=N^*$ be the dual lattice.
The multiplicative group of units of the torus is 
$$
\Gamma(T_N,\cO_{T_N}^*)=\{c \chi^m;c\in k^\times, m\in M\}.
$$
Let $v\in M$. Let $n\ge 1$. Denote $M'=M+\Z \frac{v}{n}\subset M_\Q$.
The lattice dual to $M'$ is 
$$
N'=\{e\in N;\langle v,e\rangle\in n\Z\}.
$$
The set $\{i\in \Z;i\frac{v}{n}\in M\}$ is a subgroup of $\Z$, of the form
$n'\Z$ for some divisor $1\le n'\mid n$. Denote $d=n/n'$ and $v'=v/d\in M$. 
Since $\frac{v}{n}=\frac{v'}{n'}$, we obtain $\frac{i}{n'}v\notin M$ for every $0<i<n'$. 

\begin{prop}\label{ru}
\begin{itemize}
\item[a)] Suppose $n'=n$. Then the $n$-th root of the unit 
$\chi^v\in  \Gamma(T_N,\cO^*_{T_N})$ is the torus homomorphism 
$
T_{N'}\to T_N
$ 
induced by the inclusion $N'\subseteq N$, endowed with the unit
$\chi^{\frac{v}{n}}\in  \Gamma(T_{N'},\cO^*_{T_{N'}})$.
\item[b)] Suppose $k$ contains $\mu_n$, the cyclic group of order $n$.
Then $T_N[\sqrt[n]{\chi^v}]$ is reduced, with $d$ irreducible components, and irreducible decomposition
$
T_N[\sqrt[n]{\chi^v}]=\sqcup_{\zeta\in \mu_d}T_N[\sqrt[n']{\zeta \chi^{v'}}].
$
Each irreducible component is isomorphic to $T_{N'}$ over $T_N$.
\end{itemize}
\end{prop}

\begin{proof} a) We check that $k[M]\subseteq k[M']$ and $\chi^v=(\chi^{\frac{v}{n}})^n$ 
satisfy the universal property of the root. Let 
$g\colon k[M]\to B$ be a ring homomorphism such that $g(\chi^v)=b^n$
for some $b\in B$. The assumption $n'=n$ is equivalent to $\frac{i}{n}v\notin M$
for every $0<i<n$. That is every element $m'\in M'$ has a unique representation
$$
m'=m+\frac{i}{n}v\ (m\in M,0\le i\le n-1).
$$
Define $g'\colon k[M']\to B$ by $g'(c\chi^{m'})=g(c\chi^m)\cdot b^i \ (c\in k)$.
This induces a ring homomorphism, the unique extension of $g$ to 
$k[M']$ such that $g(\chi^{\frac{v}{n}})=b$.

b) We have 
$
T^n-\chi^v=\prod_{\zeta\in \mu_d}(T^{n'}-\zeta \chi^{v'}).
$
By a), each factor is irreducible. Therefore
$T_N[\sqrt[n]{\chi^v}]$ is reduced, with $d$ irreducible components, and irreducible
decomposition
$$
T_N[\sqrt[n]{\chi^v}]=\cup_{\zeta\in \mu_d}T_N[\sqrt[n']{\zeta \chi^{v'}}].
$$
The union is disjoint, since $T_N$ is normal and $T_N[\sqrt[n]{\chi^v}]\to T_N$ is 
\'etale. Since $k$ contains $\mu_n$, $\sqrt[n']{\zeta}\in k$ for every 
$\zeta\in \mu_d$. By Remark~\ref{unc}, each irreducible component is isomorphic over 
$T_N$ with $T_{N'}$.
\end{proof}

\begin{rem} With the same proof, we obtain:
let $T_N\subset T_N\emb(\Delta)=X$ be a torus embedding such that 
$\chi^v$ is a regular function on $X$.
\begin{itemize}
\item[a)] Suppose $n'=n$. Then the $n$-th root of $\chi^v\in  \Gamma(X,\cO_X)$ 
is the toric morphism 
$$
X'=T_{N'}\emb(\Delta)\to X=T_N\emb(\Delta)
$$ 
induced by the inclusion $N'\subseteq N$, endowed with the regular function 
$\chi^{\frac{v}{n}}\in  \Gamma(X',\cO_{X'})$.
\item[b)] Suppose $k$ contains $\mu_n$. Then 
$X[\sqrt[n]{\chi^v}]$ is reduced, with $d$ irreducible components, and 
irreducible decomposition
$
X[\sqrt[n]{\chi^v}]=\cup_{\zeta\in \mu_d}X[\sqrt[n']{\zeta \chi^{v'}}].
$
Each irreducible component is isomorphic over $X$ with $T_{N'}\emb(\Delta)$.
\end{itemize}
\end{rem}


\subsection{Roots of units in a field}


Let $K$ be a field. Let $n$ be a positive integer which is not divisible by the characteristic of $K$.
The polynomial $T^n-1\in K[T]$ is separable. It has $n$ distinct roots in the algebraic closure of $K$,
denoted
$$
\mu_n(K)=\{x\in \bar{K};x^n=1\}.
$$
The multiplicative group $\mu_n(K)$ must be cyclic, hence isomorphic to $\Z/n\Z$.
An $n$-th root of unity $\zeta\in \mu_n(K)$ is a generator if and only if $\zeta^{n/d}\ne 1$ for every 
divisor $1<d\mid n$. A generator is called a {\em primitive $n$-th root of unity of $K$}.
We have $T^n-1=\prod_{\zeta\in \mu_n(K)}(T-\zeta)$ in $\bar{K}[T]$. Therefore
$$
T^n-x^n=\prod_{\zeta\in \mu_n(K)}(T-\zeta x)
$$
for every $x\in K^\times$. If $K$ contains $\mu_n(K)$, the decomposition holds in $K[T]$.

Let $f\in K^\times$. Consider the integral extension $K\to K[T]/(T^n-f)$. We have
$$
K[T]/(T^n-f)=\oplus_{i=0}^{n-1}Kt^i.
$$

\begin{lem}\label{de} Let $1\le d\mid n$ be the maximal divisor of $n$ such that $T^d-f$ has a 
root in $K$, say $g$. Suppose $K$ contains $\mu_d(K)$ and a root of $T^4+4$.
Let $n=dn'$. Then 
$$
T^n-f=\prod_{\zeta\in \mu_d(K)}(T^{n'}-\zeta g)
$$
is the decomposition into distinct irreducible factors in $K[T]$. The polynomial classes
$$
P_\zeta=\frac{1}{\prod_{\zeta'\in \mu_d(K)\setminus \zeta}(\zeta g-\zeta' g)} \frac{T^n-f}{T^{n'}-\zeta g}\in 
\frac{K[T]}{(T^n-f)} \ (\zeta\in \mu_d(K))
$$
induce an isomorphism of $K$-algebras
$$
\prod_{\zeta\in \mu_d(K)}\frac{K[T]}{(T^{n'}-\zeta g)}\isoto \frac{K[T]}{(T^n-f)}, 
(\alpha_\zeta)_\zeta\mapsto \sum_\zeta \alpha_\zeta P_\zeta.
$$
On the left hand side, each factor is a separable field extension of $K$. In particular, the ring
$K[T]/(T^n-f)$ is reduced; and an integral domain if and only if $d=1$.

The cyclic group $\mu_d(K)$ acts on $K[T]/(T^n-f)$, trivially on $K$ and by multiplication on $T$.
Under the isomorphism, this corresponds to the $\mu_d(K)$-action on the product given by the
partial permutations $\xi\colon (\alpha_\zeta)_\zeta\mapsto (\alpha_{\xi^{n'}\zeta})_\zeta$.
Moreover, if $K$ contains $\mu_n(K)$, then $\mu_n(K)$ acts on both sides by the same formulas, 
and the action is transitive (hence it permutes the factor fields).
\end{lem}

\begin{proof} The decomposition of $T^n-f$ is clear, and the factors are distinct. 
Suppose by contradiction that $T^{n'}-\zeta g$ is not irreducible. It follows from~\cite[Theorem III.9.16]{La65}
that there exists $x\in K$ 
such that a) $\zeta g=x^p$ for some prime $p\mid n'$; or b) $\zeta g=-4x^4$ and $4\mid n'$.
In case b), $\zeta g=(yx)^4$ where $y$ is a root of $T^4+4$ in $K$. Therefore
$\zeta g=h^{d'}$ for some $h\in K$ and $1<d'\mid n'$. Then 
$h^{dd'}=f$ and $d<dd'\mid n$, contradicting the maximality of $d$.

The standard formula for partial fractions with distinct linear factors gives
$1=\sum_{\zeta\in \mu_d(K)}P_\zeta$. We have $P_\zeta\cdot P_{\zeta'}=0$
if $\zeta\ne \zeta'$, and $P_\zeta^2=P_\zeta$. We obtain the desired isomorphism.

Consider the $\mu_d(K)$-action. Since $P_\zeta$ is idempotent, it follows that 
$(P_\zeta)^\xi=P_{\xi^{-n'}\zeta}$ for every $\xi\in \mu_d(K)$. Therefore the isomorphism
transforms this action into the partial permutations $\xi\colon 1_\zeta\mapsto 1_{\xi^{-n'}\zeta}$.
That is $\xi\colon (\alpha_\zeta)_\zeta\mapsto (\alpha_{\xi^{n'}\zeta})_\zeta$.
Note that the action is trivial if and only if $d\mid n'$.

Suppose moreover that $\mu_n(K)\subset K$. It acts on both sides by the same
formulas. Note that of $\xi\in \mu_n(K)$, then multiplication by $\xi^{-n'}$ induces a 
bijection of $\mu_d(K)$. The action is transitive, since for every $\zeta,\zeta'\in \mu_d(K)$
there exists $\xi\in \mu_n(K)$ such that $\xi^{n'}\zeta'=\zeta$.
\end{proof}

Note that any field which contains an algebraically closed field, contains
a root of $T^4+4$. And $T^4+4c^4=(T^2-2cT+2c^2)(T^2+2cT+2c^2)$.

\begin{lem}\label{rn}
Let $A\subseteq K$ be a domain, integrally closed in $K$.
Suppose $A$ contains $\mu_n(K)$, and the reciprocal of the Vandermonde 
determinant associated to some ordering of the elements of $\mu_n(K)$.
Consider the ring homomorphism $A\to K[T]/(T^n-f)$. The integral closure of $A$ is 
$$
\bar{A}=\oplus_{i=0}^{n-1} \{x\in K; x^n f^i\in A\} t^i.
$$
Moreover, under the product decomposition of Lemma~\ref{de}, $\bar{A}$
corresponds to the product of integral closures of $A$ in the factor fields,
and each of the $d$-factors has the explicit description
$$
\bar{A}_\zeta=\oplus_{j=0}^{n'-1} \{x\in K; x^n f^j\in A\}t_\zeta^j. 
$$
The integral extension $A\to \bar{A}$ is Galois, with Galois group $\mu_n(K)$.
The Galois group permutes the factors of the decomposition.
\end{lem}

\begin{proof} Each element $\alpha\in K[T]/(T^n-f)$ has a unique representation
$$
\alpha=\sum_{i=0}^{n-1}x_it^i \ (x_i\in K).
$$

{\em Step 1:} We claim that $\alpha\in \bar{A}$ if and only if $x_it^i\in \bar{A}$, 
for every $i$. Indeed, the converse is clear. Suppose $\alpha\in \bar{A}$. 
The cyclic group $\mu_n(K)$ acts on $K[T]/(T^n-f)$, trivially on 
$K$, and by multiplication on $T$. It also acts on $A$, since $A$ contains 
$\mu_n(K)$. Therefore $\alpha^\zeta=\sum_{i=0}^{n-1}\zeta^ix_it^i\in \bar{A}$, 
for every $\zeta\in \mu_n(K)$. Choose an ordering 
$\zeta_0,\zeta_1,\ldots,\zeta_{n-1}$ of $\mu_n(K)$. Since the Vandermonde 
determinant 
$$
\det\begin{bmatrix}
\zeta_0^0 & \zeta_0^1 & \cdots & \zeta_0^{n-1} \\
\zeta_1^0 & \zeta_1^1 & \cdots & \zeta_1^{n-1} \\
\vdots      & \vdots       & \ddots & \vdots \\
\zeta_{n-1}^0 & \zeta_{n-1}^1 & \cdots & \zeta_{n-1}^{n-1}
\end{bmatrix}
=\prod_{0\le i<j\le n-1}(\zeta_j-\zeta_i)
$$
is a unit in $A$, each $x_it^i$ is a combination of the $\alpha^\zeta$'s, with 
coefficients in $A$. Therefore $x_it^i\in \bar{A}$.

{\em Step 2:}  Let $x\in K$ and $0\le i<n$. We claim that $xt^i\in \bar{A}$ 
if and only if $x^nf^i\in A$. Indeed, denote $y=x^nf^i\in K$. We have $y=(xt^i)^n$.
If $y\in A$, then $xt^i\in \bar{A}$. Conversely, suppose 
there exists an equation in $K[T]/(T^n-f)$
$$
(xt^i)^m+\sum_{l=1}^m a_l (xt^i)^{m-l}=0\ (a_l \in A).
$$
Raising the equation to some power, we may suppose $n\mid m$. 
Then $m-l\equiv 0\pmod{n}$ if and only if $n\mid l$. So
the coefficient of $t^0$ in the equation is 
$$
y^{\frac{m}{n}}+\sum_{n\mid l}a_l y^{\frac{m-l}{n}}=0.
$$ 
Then $y$ is algebraic over $A$. Therefore $y\in A$.

{\em Step 3:} Consider the product decomposition from Lemma~\ref{de}.
We have $(\sum_\zeta \alpha_\zeta P_\zeta)^m=\sum_\zeta \alpha^m_\zeta P_\zeta$
for every $m\ge 0$. Therefore $\sum_\zeta \alpha_\zeta P_\zeta$ is integral over $A$
if and only if each $\alpha_\zeta$ is integral. By Steps 1 and 2, the integral closure 
of $A$ in $K[T]/(T^{n'}-\zeta g)$ is the integrally closed domain
$$
\bar{A}_\zeta=\oplus_{j=0}^{n'-1}\{x\in K; x^{n'}(\zeta g)^j\in A\}t_\zeta^j.
$$
Since $A$ is integrally closed in $K$, $x^{n'}(\zeta g)^j\in A$ if and only if 
$(x^{n'}(\zeta g)^j)^d\in A$. That is $x^nf^q\in A$. Therefore 
$$
\bar{A}_\zeta=\oplus_{j=0}^{n'-1}\{x\in K; x^nf^j\in A\}t_\zeta^j.
$$
In particular, the product decomposition of Lemma~\ref{de} induces an isomorphism
of $A$-algebras
$$
\prod_{\zeta\in \mu_d(K)} \bar{A}_\zeta \isoto \bar{A}.
$$
\end{proof}

\begin{rem}\label{par}
Denote $X=\Spec A$. If $f$ is a non-zero rational function on $X$, that
is $f\in Q(A)$, then
$
\{x\in Q(A);x^nf^i\in A\}=\Gamma(X,\cO_X(\lfloor \frac{i}{n}\dv(f)\rfloor)).
$
\end{rem}


\subsection{Irreducible components, normalization of roots}


\begin{lem}\label{inc}
Let $A$ be an integral domain, integrally closed in its field of fractions $Q$. Suppose $Q$ contains
a root of $T^4+4$. Let $0\ne a\in A$ and $n\ge 1$.
\begin{itemize}
\item[1)] Suppose $T^d-a$ has no root in $A$, for every divisor $1<d\mid n$. Then $A[T]/(T^n-a)$ 
is an integral domain with quotient field $Q[T]/(T^n-a)$.
\item[2)] Suppose $\Char(Q)\nmid n$, and $A$ contains $\mu_n(Q)$ and the reciprocal of the
Vandermonde determinant associated to some ordering of the elements of $\mu_n(Q)$ 
(e.g. $A$ contains an algebraically closed field whose characteristic does not divide $n$).
Then the ring $A[T]/(T^n-a)$ has no nilpotents, and the integral closure in its total ring of fractions 
$Q[T]/(T^n-a)$ is 
$
\oplus_{i=0}^{n-1}\{q\in Q;q^na^i\in A\}t^i.
$
\end{itemize}
\end{lem}

\begin{proof} 1) Since $A$ is integrally closed, $T^d-a$ has a root in $A$ if and only
if it has a root in $Q$. Therefore $d=1$ in Lemma~\ref{de}. Therefore $Q[T]/(T^n-a)$
is a field. The application $A[T]/(T^n-a)\to Q[T]/(T^n-a)$ is injective. Therefore 
$A[T]/(T^n-a)$ is an integral domain with quotient field $Q[T]/(T^n-a)$.

2) Let $1\le d\mid n$ be the maximal divisor such that $a={a'}^d$ for some $a'\in A$. 
Denote $n'=n/d$. We have
$$
T^n-a=\prod_{\zeta\in \mu_d(Q)}(T^{n'}-\zeta a')
$$
By the maximality of $d$, each $T^{n'}-\zeta a'$ is irreducible in $Q[T]$. Since $A^\times$
contains $\mu_d$, there are no multiple factors. Therefore $A[T]/(T^n-a)$ is reduced, 
with irreducible components $A[T]/(T^{n'}-\zeta a')$. 

Finally, since $a\in A$, the homomorphism $A\to Q[T]/(T^n-a)$ factors as 
$A\to A[T]/(T^n-a)\to Q[T]/(T^n-a)$. The ring extension $A\to A[T]/(T^n-a)$ is 
integral. Therefore the integral closure of $A$ in $Q[T]/(T^n-a)$,
computed in Lemma~\ref{rn}, coincides with the integral closure of $A[T]/(T^n-a)=A[\sqrt[n]{a}]$
in its total ring of fractions $Q[T]/(T^n-a)$.
\end{proof}

\begin{prop}\label{ic} Let $X$ be normal and irreducible scheme.
Suppose the field of fractions $Q(X)$ contains a root of $T^4+4$.
Let $n\ge 1$ and $s\in \Gamma(X,\cL^n)$.
\begin{itemize}
\item[a)] Suppose $T^d-s$ has no root in $\Gamma(X,\cL^{\frac{n}{d}})$, for every divisor $1<d\mid n$.
Then $X[\sqrt[n]{s}]$ is reduced and irreducible. 
\item[b)] Let $1\le d\mid n$ be the maximal divisor of $n$ with the property that $T^d-s$ 
has a root in $\Gamma(X,\cL^{\frac{n}{d}})$, say $s'$. Suppose $\Gamma(X,\cO_X^\times)$ 
contains $\mu_d$. Then $X[\sqrt[n]{s}]$ is reduced, with irreducible decomposition
$$
X[\sqrt[n]{s}]=\cup_{\zeta\in \mu_d}X[\sqrt[\frac{n}{d}]{\zeta s'}].
$$
\item[c)] Suppose $s\ne 0$. Let $U\subseteq X$ be a non-empty open subset such that $\cL|_U$
has a nowhere zero section $u$. Let $s|_U=fu^n$ with $f\in \Gamma(U,\cO_X)$.
The normalization of $X[\sqrt[n]{s}]$ coincides with the normalization of
$X$ in the ring extension $Q(X)\to Q(X)[T]/(T^n-f)$.
\end{itemize}
\end{prop}

\begin{proof} 
a) Let $U\subset X$ be a non-empty affine open subset such that $\cL|_U$ has a nowhere zero section $u$. 
We claim that $T^d-s|_U$
has no root in $\Gamma(U,\cL^{\frac{n}{d}})$, for every divisor $1<d\mid n$.
Indeed, if $s_d\in \Gamma(U,\cL^{\frac{n}{d}})$ is a root, then since $s_d^d=s
\in \Gamma(X,\cL^n)$, it follows that $s_d$ is the restriction to $U$ of some 
$s'\in \Gamma(X,\cL^{\frac{n}{d}})$. Then ${s'}^d$ and $s$ coincide
on the dense open subset $U$, hence they are equal. This contradicts our
assumption.

Let $s|_U=fu^n$ with $f\in A=\Gamma(U,\cO_X)$. The claim is equivalent to the 
following property: $T^d-f$ has no root in $A$,
for every divisor $1<d\mid n$. By Lemma~\ref{inc}, $\pi^{-1}(U)$ is reduced
and irreducible, and dominates $U$.

Since $X$ can be covered by subsets $U$ as above, it follows that 
$X[\sqrt[n]{s}]$ is reduced and irreducible.

b) We have $T^n-s=T^n-{s'}^d=\prod_{\zeta\in \mu_d}(T^\frac{n}{d}-\zeta s')$.
The factors are distinct.
Since $d$ is maximal, $T^{d'}-\zeta s'$ has no root in $\Gamma(X,\cL^{\frac{n}{dd'}})$,
for every divisor $1<d'\mid \frac{n}{d}$. By a),
$X[\sqrt[\frac{n}{d}]{\zeta s'}]$ are the irreducible components of $X[\sqrt[n]{s}]$.

c) This follows from Lemma~\ref{inc}.
\end{proof}

From Lemma~\ref{inc} and Remark~\ref{par}, we deduce

\begin{prop}\label{2.17}
Let $k$ be an algebraically closed field. Let $n$ be a positive integer which is not
divisible by $\Char(k)$. Let $X/k$ be a normal algebraic variety, let 
$\cL$ be an invertible $\cO_X$-module, and $0\ne s\in \Gamma(X,\cL^n)$. 
Let $D$ be the zero locus of $s$, an effective Cartier divisor on $X$. 
Denote by $\bar{X}[\sqrt[n]{s}]\to X[\sqrt[n]{s}]$ the normalization, and 
$\bar{\pi}\colon \bar{X}[\sqrt[n]{s}]\to X$ the induced morphism. 
$$
\xymatrix{
X[\sqrt[n]{s}]   \ar[d]_\pi &     \bar{X}[\sqrt[n]{s}]  \ar[l] \ar[dl]^{\bar{\pi}} \\
 X                             &  
}
$$ 
Then $\bar{\pi}$ is a Galois ramified cover, with Galois group cyclic of order $n$, 
and eigenspace decomposition
$$
\bar{\pi}_*\cO_{\bar{X}[\sqrt[n]{s}]}=\oplus_{i=0}^{n-1}\cL^{-i}(\lfloor \frac{i}{n}D\rfloor).
$$
\end{prop}

Call $\bar{\pi}$ the {\em cyclic cover obtained by taking the $n$-th root out of $s$}.
Note that $\bar{\pi}$ is flat if and only if the Weil divisor $\lfloor \frac{i}{n}D\rfloor$ is 
Cartier, for every $0<i<n$. If $X/k$ is smooth, Proposition~\ref{2.17} was proved in
~\cite[Corollary 3.11]{EVlect}.


\section{Normalized roots of rational functions}


Let $n$ be a positive integer, let $k$ be an algebraically closed field which 
contains $n$ distinct roots of unity.
We consider the category of normal $k$-varieties and dominant morphisms.
A morphism $f\colon X\to Y$ is called {\em dominant} if for every irreducible component
$X_i$ of $X$, $\overline{f(X_i)}$ is an irreducible component of $Y$. A composition of
dominant morphisms is again dominant. But the image of a dominant morphism may
not be dense (e.g. the open embedding of a connected component).

Let $X,Y$ be normal $k$-varieties. A dominant morphism $f\colon X\to Y$ induces 
a pullback homomorphism of rings $f^*\colon k(Y)\to k(X)$. It is compatible with
composition of dominant morphisms, and maps invertible rational functions to invertible 
rational functions. Note that $f^*$ may not be injective.

\begin{prop} Let $X/k$ be a normal variety, $\varphi$ an invertible
rational function on $X$.
Then there exists a normal variety $Y/k$, a dominant morphism $\pi\colon Y\to X$ 
and an invertible rational function $t$ on $Y$, such that $t^n=\pi^*\varphi$, 
and the following universal property holds: if $Y'/k$ is a normal variety,
$g\colon Y'\to X$ is a dominant morphism, and $t'$ is an invertible rational 
function on $Y'$ such that ${t'}^n=g^*\varphi$, then there exists a unique 
dominant morphism $u\colon Y'\to Y$ such that $g=\pi\circ u$ and $t'=u^*t$.
\end{prop}

\begin{proof}
{\em Step 1:} Suppose $X=\Spec A$ and $A$ is an integral domain. 
Then $\varphi \in Q(A)$. Let $\bar{A}$ be the integral closure of $A$ in the ring extension 
$$
A\to \frac{Q(A)[T]}{(T^n-\varphi)}.
$$ 
Then $A\to \bar{A}$ induces a finite morphism $\pi\colon Y\to X$. By Lemmas~\ref{de} 
and~\ref{rn}, $Q(A)[T]/(T^n-\varphi)$ is a product of fields $K_\zeta$, and $\bar{A}$ is 
the product of the normalization of $A$ in $K_\zeta$. Each factor is a 
integral domain, integrally closed in its function field. Therefore $Y$ is normal. The morphism
$\pi$ is dominant since it is finite. If we denote by $t\in k(Y)$ the class of $T$, then $t$
is invertible and $t^n=\pi^*\varphi$.
 
Let $Y'/k$ be a normal variety, $g\colon Y'\to X$ a dominant morphism, and $t'$ an 
invertible rational function on $Y'$ such that ${t'}^n=g^*\varphi$. There exists a unique homomorphism of $Q(A)$-algebras
$$
\frac{Q(A)[T]}{(T^n-\varphi)}\to Q(Y')
$$
which maps $T$ to $t'$. Since $A$ maps to $\Gamma(Y',\cO_{Y'})$,
$\bar{A}$ maps to the integral closure of $\Gamma(Y',\cO_{Y'})$ in $Q(Y')$.
Since $Y'$ is normal, $\Gamma(Y',\cO_{Y'})$ is integrally closed in $Q(Y')$.
We obtain a morphism $\bar{A}\to \Gamma(Y',\cO_{Y'})$. This translates into a 
dominant morphism $u\colon Y'\to Y$ with $g=\pi\circ u$ and $t'=u^*t$.

{\em Step 2:} Consider the $\cO_X$-algebra  
$$
\cA(X,\varphi,n)=\oplus_{i=0}^{n-1}\cO_X(\lfloor \frac{i}{n}\dv(\varphi)\rfloor),
$$
with the following multiplication: if $u_i,u_j$ are local sections of $\cA_i$ and $\cA_j$ respectively,
their product is the rational function $u_iu_j\in \cA_{i+j}$ if $i+j<n$, and $u_i u_j\varphi\in \cA_{i+j-n}$ if $i+j\ge n$. Let 
$\pi\colon Y=\Spec_X \cA \to X$ be the induced finite morphism of schemes.

Let $U=\Spec A\subset X$ be an affine irreducible open subset.
By Remark~\ref{par}, 
$$
\Gamma(U,\cA_i)=\{\psi\in Q(A);\psi^n\varphi^i\in A\}.
$$
By Lemma~\ref{rn}, the integral closure of $A$ in $Q(A)[T]/(T^n-\varphi)$ is
$$
\bar{A}= \oplus_{i=0}^{n-1}\{\psi\in Q(A); \psi^n\varphi^i\in A\}t^i.
$$
Therefore $\psi_i\mapsto \psi_i t^i$ induces an isomorphism of $A$-algebras
$$
\Gamma(U,\cA)\isoto \bar{A}.
$$
Therefore the construction of $\cA(X,\varphi,n)$ globalizes the local construction in 
Step 1. Since $X$ is covered by such $U$, we deduce that $Y$ is normal.

Let $V=X\setminus \Supp(\varphi)$ be the open dense subset of $X$ where $\varphi$
is a unit. Then $\Gamma(V,\cA_1)=\Gamma(V,\cO_V)$. Let $t'=1\in \Gamma(V,\cA_1)$.
Then ${t'}^n=\varphi$ in $\Gamma(V,\cA)$. So $t'$ becomes an invertible rational
function on $Y$ such that ${t'}^n=\pi^*\varphi$.
The universal property can be checked on irreducible affine open subsets of $X$, 
so it follows from Step 1.
\end{proof}

The morphism $\pi\colon Y\to X$, endowed with the invertible rational function 
$t\in k(Y)$, is unique up to an isomorphism over $X$. It is called the {\em 
normalized $n$-th root of $X$ with respect $\varphi$}. We denote 
$Y$ by $X[\varphi,n]$, and $t$ by $\sqrt[n]{\varphi}$. The $\Q$-Weil divisor 
$D=\frac{1}{n}\dv(\varphi)$ satisfies $nD\sim 0$. 

\begin{lem}\label{bnr} The following properties hold:
\begin{itemize}
\item[a)] $\pi_*\cO_{X[\varphi,n]}=\oplus_{i=0}^{n-1}\cO_X(\lfloor iD \rfloor)$.
\item[b)] The cyclic group $\Z/n\Z$ acts faithfully on $X[\varphi,n]$, $\pi$ 
is the induced quotient morphism, and a) is the decomposition into eigensheaves.
\item[c)] The morphism $\pi$ is finite. It is flat if and only if the Weil 
divisors $\lfloor D \rfloor,\ldots,\lfloor (n-1)D \rfloor$ are Cartier.
\item[d)] Let $\tau\colon X'\to X$ be an \'etale morphism. Then the normalized
$n$-th root of $X'$ with respect to $\tau^*\varphi$ is the pullback morphism 
$X[\varphi,n]\times_X X'\to X'$, endowed with the pullback invertible rational function.
\end{itemize}
\end{lem}

\begin{proof} Properties a), b), c) follow from the local description of the normalized
$n$-th root. In b), let $\zeta\in k^*$ be a primitive $n$-th root of $1$. The action is 
$(\zeta,aT^i)\mapsto \zeta^i a T^i$ for the local model in Step 1, 
and $(\zeta,u_i)\mapsto \zeta^iu_i$ for the global model in Step 2.
For d), Lemma~\ref{eds} gives $\tau^*\cA(X,\varphi,n)\isoto 
\cA(X',\tau^*\varphi,n)$. Therefore the normalized $n$-th root commutes
with \'etale base change.
\end{proof}

\begin{lem}\label{invun}
Let $\varphi'$ be another invertible rational function on $X$. Then 
$X[\varphi,n]$ is naturally isomorphic to $X[{\varphi'}^n\varphi,n]$ over $X$.
\end{lem}

\begin{proof} For an invertible $\varphi\in k(X)$ and a Weil divisor $D$ on $X$, the following
formula holds: 
$$
\cO_X((\varphi)+D)=\varphi^{-1}\cO_X(D).
$$
Therefore the application $s_i\mapsto {\varphi'}^is_i$ induces an isomorphism of 
$\cO_X$-algebras
$$
\cA(X,{\varphi'}^n\varphi,n)\isoto \cA(X,\varphi,n).
$$
\end{proof}

Suppose $X$ is irreducible. Let $1\le d\mid n$ be the maximal divisor such that 
$\varphi=\psi^d$ for some $\psi\in k(X)^*$. Then $X[\varphi,n]$ has
exactly $d$ irreducible (connected) components
$$
X[\varphi,n]=\sqcup_{\zeta\in \mu_d} X[\zeta\psi,n/d].
$$ 
Each component is isomorphic over $X$ to $X[\psi,n/d]$.


\subsection{Structure in codimension one}

At the generic point of each prime divisor on $X$, the normalized root is explicitly
described by the following lemma. We use the convention $\gcd(n,0)=n$.

\begin{lem}\label{exp}
Suppose $\varphi=uf^m$, where $u\in \Gamma(X,\cO_X^*)$, $f\in \Gamma(X,\cO_X)$ 
is a non-zero divisor such that the divisor $\dv(f)$ is reduced, and $m\in \Z$. 
Let $g=\gcd(n,m)$. Let $n=gn'$, and $1\le j\le n'$ with $jm\equiv g \pmod n$. Then 
there exists an isomorphism of $\cO_X$-algebras
$$ 
\pi_*\cO_Y\simeq \cO_X[T_1,T_2]/(T_1^g-u,T_2^{n'}-fT_1^j).
$$
That is, $\pi$ is isomorphic to the composition of the $g$-th root of the unit $u$, followed
by the $n'$-th root of the regular function $\sqrt[g]{u}^j f$. The above formula simplifies to 
$\pi_*\cO_Y\simeq \cO_X[T]/(T^n-u^jf)$ if $g=1$,
and to 
$\pi_*\cO_Y\simeq \cO_X[T]/(T^n-u)$ if $g=n$.
\end{lem}

\begin{proof} Since $\dv(f)$ is reduced, we have $\lfloor iD \rfloor=\lfloor \frac{im}{n}\rfloor\dv(f)$. Therefore
$$
\cA=\oplus_{i=0}^{n-1} \cO_X f^{-\lfloor \frac{im}{n}\rfloor } t^i \ (t^n=uf^m).
$$
Let $m=gm'$. Given $0\le i<n$, there are uniquely defined integers $0\le \alpha<n',0\le \beta<g$ such that
$i \equiv \alpha j \pmod{n'}$ and $\frac{i-\alpha j}{n'} \equiv \beta \pmod g$.
In particular, $\{\frac{jm}{n} \}=\frac{1}{n'}$ and $\{\frac{jm\alpha}{n} \}=\frac{\alpha}{n'}$.
Let $\gamma\in \Z$ with $i-\alpha j-n'\beta=n\gamma$. We obtain
$
m\gamma  + \lfloor \frac{jm}{n}\rfloor \alpha+m' \beta = \lfloor \frac{im}{n}\rfloor.
$
Therefore the following holds in $\cA$:
$$
u^\gamma (f^{-\lfloor \frac{jm}{n}\rfloor} t^j)^\alpha (f^{-m'} t^{n'})^\beta=f^{-\lfloor \frac{im}{n} \rfloor} t^i.
$$
It follows that the homomorphism
$$
\cO_X[T_1,T_2]/(T_1^g-u,T_2^{n'}-fT_1^j ) \to \cA, \ 
T_1\mapsto f^{-m'} t^{n'}, T_2\mapsto  f^{-\lfloor \frac{jm}{n}\rfloor} t^j
$$
is well defined and surjective. It is injective by
the uniqueness of $\alpha,\beta$. The simplified forms
of the formula are clear.
\end{proof}

\begin{lem}\label{1.4} The ramification index of $\pi$ over a prime divisor $E$ of $X$ is 
$$
\frac{n}{\gcd(n,\ord_E(\varphi))}.
$$ 
In particular, $\pi$ ramifies exactly over the prime divisors of $\Supp\{D\}$. Moreover,
$$
\pi^*D=\sum_{E'\subset Y}\frac{\ord_{\pi(E')}(\varphi)}{\gcd(n,\ord_{\pi(E')}(\varphi))}E'.
$$
\end{lem}

\begin{proof} Let $E$ be a prime divisor on $X$. Let $m=\ord_E(\varphi)$. Let $U$ be an 
open subset such that $U\cap E\ne \emptyset$ is nonsingular, cut out by a local 
parameter $f\in \Gamma(U,\cO_X)$. By shrinking $U$, we may suppose $\varphi=uf^m$, 
for some $u\in \Gamma(U,\cO_X^\times)$.  
Let $g=\gcd(n,m)$, $n=gn'$. By Lemma~\ref{exp}, $\pi^{-1}(U)\to U$ is the composition of
the \'etale cover $U[\sqrt[g]{u}]\to U$, followed by the $n'$-th root of the regular function
$\sqrt[g]{u}^jf$, which is a local parameter at each prime of $U[\sqrt[g]{u}]$ over $E$.
Therefore the ramification index over $E$ is $n'$. The pullback formula follows. 
Note that $n'\ne 1$ if and only if $\mult_E D\notin \Z$.
\end{proof}

\begin{lem}\label{dvc} Let $\Sigma$ be a reduced Weil divisor on $X$ which contains 
$\Supp \{D\}$. Let $\Sigma_Y=\pi^{-1}\Sigma$ be the preimage reduced
Weil divisor. We have eigenspace decompositions:
\begin{itemize}
\item[a)] $\pi_*\tilde{\Omega}^p_{Y/k}(\log \Sigma_Y)=\oplus_{i=0}^{n-1}\tilde{\Omega}^p_{X/k}(\log \Sigma)
(\lfloor iD\rfloor)$.
\item[b)] $\pi_*\tilde{\Omega}^p_{Y/k}=\oplus_{i=0}^{n-1}\tilde{\Omega}^p_{X/k}(\log \Supp \{iD\})(\lfloor iD\rfloor)$.
\item[c)] $\pi_*\cT_{Y/k}=\oplus_{i=0}^{n-1}\tilde{\cT}_{X/k}(-\log \sum_E \epsilon(d_E,i)E)(\lfloor iD\rfloor)$,
where for a rational number $d$ we define $r(d)=\min\{r\ge 1;rd\in \Z\}$, and 
$\epsilon(d,i)$ to be $1$ if $d\notin \Z$ and $id+\frac{1}{r(d)}\notin \Z$,
and zero otherwise. In particular, the invariant part of $\pi_*\cT_{Y/k}$ is $\tilde{\cT}_{X/k}(-\log \Supp\{D\})$.
\item[d)] $\pi_*\cT_{Y/k}(-\log \Sigma_Y)=\oplus_{i=0}^{n-1}\tilde{\cT}_{X/k}(-\log \Sigma)(\lfloor iD\rfloor)$.
\end{itemize}
\end{lem}

\begin{proof} Let $V=X\setminus (\Sing X\cup \Supp\{D\})$, an open dense subset of $X$.
Then $\pi$ is \'etale over $V$. In particular, $\pi^{-1}(V)$ is also nonsingular, and 
$\pi^*\Omega_{V/k}\simeq \Omega_{\pi^{-1}(V)/k}$. Therefore $\pi^*\Omega^p_{V/k}\simeq 
\Omega^p_{\pi^{-1}(V)/k}$, and the projection formula gives 
$$
\pi_*\Omega^p_{\pi^{-1}(V)/k}=\oplus_{i=0}^{n-1}\Omega^p_{V/k}(iD|_V) t^i.
$$
This describes the sheaves in a) and b) at the generic points of $X$. These sheaves are $S_2$,
so we may determine them locally near a fixed prime divisor on $X$. Let $E$ be a prime divisor on $X$. 
We may shrink $X$ and suppose $\varphi=uf^m$, with $u$ a unit and $f$ a parameter for $E$.
Let $E'$ be a prime divisor over $E$. Then $t_2$ is a local parameter at $E'$, and $t_1$ is a unit at $E'$
(in the notations of Lemma~\ref{exp}). Recall that $t$ is the $n$-th root of $\varphi$.
 
Let $\omega$ be a rational $p$-form on $X$. There exists a unique integer $a$ such that 
$$
f^a\omega=\frac{df}{f}\wedge \omega^{p-1}+\omega^p,
$$ 
with $\omega^{p-1},\omega^p$ rational forms which are regular at $E$, and $\omega^{p-1}|_E\ne 0$. 
From $f=t_1^{_j}t_2^{n'}$, we obtain
$$
\frac{df}{f}=n'\frac{dt_2}{t_2}-j\frac{dt_1}{t_1}.
$$
Therefore $(f^at^{-i}) (\omega t^i)-n'\frac{dt_2}{t_2}\wedge \omega^{p-1}$ is regular at $E'$, and
$\omega^{p-1}|_{E'}\ne 0$. Since $(f^at^{-i})^n$ and $t_2^{n'(na-mi)}$ differ by a unit at $E'$,
the order of $f^at^{-i}$ at $E'$ if $n'(a-\frac{mi}{n})$. Therefore 
$\omega t^i$ has at most a logarithmic pole at $E'$ if and only if $n'(a-\frac{mi}{n})\le 0$,
if and only if $a\le \lfloor \frac{im}{n}\rfloor$, if and only if $f^{\lfloor\frac{im}{n}\rfloor} \omega$
has at most a logarithmic pole at $E$. This proves a). Similarly, 
$\omega t^i$ is regular at $E'$ if and only if $n'(a-\frac{mi}{n})<0$,
if and only if $a<\frac{im}{n}$. That is $a\le \lfloor \frac{im}{n}\rfloor$ if $\frac{im}{n}\notin \Z$,
and $a<\frac{im}{n}$ if $\frac{im}{n}\in \Z$. That is $f^{\lfloor\frac{im}{n}\rfloor} \omega$ has at most
a logarithmic pole at $E$ if $\frac{im}{n}\notin \Z$, and is regular at $E$ if $\frac{im}{n}\in \Z$. This
proves b).

Since $\pi$ is \'etale over $V$, every $k$-derivation $\theta$ of $V$ lifts to 
a unique $k$-derivation $\tilde{\theta}$ of $\pi^{-1}(V)$. We have an eigenspace decomposition
$$
\pi_*\cT_{\pi^{-1}(V)/k}=\oplus_{i=0}^{n-1}\{\tilde{\theta};\theta\in \cT_{V/k}(iD|_V)\}  t^i.
$$
This determines the sheaves in c) and d) at the generic points of $X$. The sheaves are $S_2$,
so we may localize near the generic point of a prime divisor $E$ of $X$. We use the same notations
as above. Let $A\to \bar{A}$ be the integral extension obtain by localizing $\pi$ at $E$.
We compute
$$
\bar{A}=\oplus_{i=0}^{n-1}Af^{-\lfloor \frac{im}{n} \rfloor}t^i.
$$
Let $\theta$ be a rational $k$-derivation of $A$. From $t^n=\varphi=uf^m$, we obtain 
$$
\tilde{\theta}(af^{-\lfloor \frac{im}{n} \rfloor}t^i)=(\theta(a)+a \frac{i}{n} \frac{\theta(u)}{u}+a\{\frac{im}{n}\}
\frac{\theta(f)}{f}) f^{-\lfloor \frac{im}{n} \rfloor}t^i \ (a\in A).
$$

c) Let $0\le l\le n-1$. We claim that $\tilde{\theta}f^{-\lfloor \frac{lm}{n} \rfloor}t^l$ is a regular at $E'$ if 
and only if $\theta$ is regular at $E$, and moreover logarithmic at $E$ in case $\frac{m}{n}\notin \Z$ 
and $\{ \frac{lm}{n} \}\ne 1-\frac{1}{n'}$.

Indeed, the rational derivation $\tilde{\theta}f^{-\lfloor \frac{lm}{n} \rfloor}t^l$ is 
regular on $\bar{A}$ if and only if for every $a\in A$ and $0\le i\le n-1$, 
$$
\tilde{\theta}(af^{-\lfloor \frac{im}{n} \rfloor}t^i)f^{-\lfloor \frac{lm}{n} \rfloor}t^l
\subseteq A  f^{-\lfloor \frac{(i+l\mod n)m}{n} \rfloor}t^{i+l\mod n}.
$$
Since $f^{-\lfloor \frac{i(i+l\mod n)m}{n} \rfloor}t^{i+l\mod n}$ differs by 
$f^{-\lfloor \frac{(i+l)m}{n} \rfloor}t^{i+l}$ by a unit, the condition becomes 
$$
\tilde{\theta}(af^{-\lfloor \frac{im}{n} \rfloor}t^i)f^{-\lfloor \frac{lm}{n} \rfloor}t^l
\subseteq A  f^{-\lfloor \frac{(i+l)m}{n} \rfloor}t^{i+l}.
$$
From above, this is equivalent to 
$$
\theta(a)+a \frac{i}{n} \frac{\theta(u)}{u}+a\{\frac{im}{n}\} \frac{\theta(f)}{f} \in 
Af^{\lfloor \frac{im}{n} \rfloor+ \lfloor \frac{lm}{n} \rfloor-\lfloor \frac{(i+l)m}{n} \rfloor}.
$$
Set $i=0$. The condition becomes $\theta(a)\in A$. That is $\theta\in \Der_k(A)$.
With this assumption, the condition becomes 
$$
\{\frac{im}{n}\} \frac{\theta(f)}{f} \in Af^{\lfloor \frac{im}{n} \rfloor+ \lfloor \frac{lm}{n} 
\rfloor-\lfloor \frac{(i+l)m}{n} \rfloor} \ (0\le i\le n-1).
$$
Note $\lfloor \frac{im}{n} \rfloor+ \lfloor \frac{lm}{n} \rfloor-\lfloor \frac{(i+l)m}{n} \rfloor$ is 
$-1$ or $0$. The latter case happens if and only if $\{ \frac{im}{n} \}+ \{ \frac{lm}{n} \}<1$.
If $\frac{m}{n}\in \Z$, this always holds. Else, let $\frac{m}{n}=\frac{m'}{n'}$ be the reduced
form, with $n'>1$. Suppose $\{ \frac{lm}{n} \}=1-\frac{1}{n'}$. If $\{ \frac{im}{n} \}+ \{ \frac{lm}{n} \}<1$,
then $\{ \frac{im}{n} \}=0$, so the condition holds again. Suppose $\{ \frac{lm}{n} \}\ne 1-\frac{1}{n'}$.
Equivalently, $\{ \frac{lm}{n} \}<1-\frac{1}{n'}$.
Recall from Lemma~\ref{exp} that $\{ \frac{jm}{n} \}=\frac{1}{n'}$. The condition for $i=j$ becomes 
$
\frac{\theta(f)}{f} \in A.
$

d) We claim that $\tilde{\theta}f^{-\lfloor \frac{lm}{n} \rfloor}t^l$ is a regular and logarithmic at 
$E'$ if and only if $\theta$ is regular and logarithmic at $E$. Indeed, a local parameter at $E'$ is 
$t_2=f^{-\lfloor \frac{jm}{n}\rfloor}t^j$
(recall $1\le j\le n, jm\equiv g\mod n$). We compute
\begin{align*}
\frac{\tilde{\theta}f^{-\lfloor \frac{lm}{n} \rfloor}t^l(t_2)}{t_2} & =
(\frac{j}{n}\frac{\theta(u)}{u}+ \{ \frac{jm}{n} \} \frac{\theta(f)}{f})f^{-\lfloor \frac{lm}{n} \rfloor}t^l  \\
 & =
(\frac{j}{n}\frac{\theta(u)}{u}+\frac{1}{n'} \frac{\theta(f)}{f})f^{-\lfloor \frac{lm}{n} \rfloor}t^l
 \end{align*}
 The last term is regular at $E'$ if and only if 
 $\frac{j}{n}\frac{\theta(u)}{u}+\frac{1}{n'} \frac{\theta(f)}{f} \in A$.
 That is $\frac{\theta(f)}{f} \in A$. From above, the claim holds.
\end{proof}


\subsection{Toroidal criterion}


\begin{prop}\label{nru} With the notations of section 2-A, let $X=T_N\emb(\Delta)$ be a 
torus embedding. The torus character $\chi^v$
becomes a rational function $\varphi$ on $X$. 
\begin{itemize}
\item[a)] Suppose $n'=n$. Then the normalized $n$-th root of $\varphi$ 
is the toric morphism 
$$
X'=T_{N'}\emb(\Delta)\to X=T_N\emb(\Delta)
$$ 
induced by the inclusion $N'\subseteq N$, endowed with the rational function 
$\chi^{\frac{v}{n}}$ on $X'$.
\item[b)] Suppose $k$ contains $\mu_n$. Then the normalized $n$-th root of $\varphi$ has 
$d$ irreducible components, each of them isomorphic over $X$ to the toric morphism
$
T_{N'}\emb(\Delta)\to T_N\emb(\Delta),
$ 
induced by the inclusion $N'\subseteq N$.
\end{itemize}
\end{prop}

\begin{proof} Let $\pi\colon Y\to X$ be the normalized root of $\varphi$.
Then $\pi\colon \pi^{-1}(T_N)\to T_N$ is the normalized $n$-th root of $\varphi|_{T_N}$.
Since $\varphi|_{T_N}$ is a unit, this coincides with the $n$-th root of $\varphi|_{T_N}$,
which is described by Proposition~\ref{ru}.
In case a), $\pi^{-1}(T_N)\to T_N$ is isomorphic over $T_N$ with $T_{N'}\to T_N$.
Therefore $Y$ is the normalization of $X$ with respect to the field extension 
$k(T_N)\to k(T_{N'})$, which is exactly $T_{N'}(\emb\Delta)\to T_N\emb(\Delta)$.
A similar argument works in case b), for each irreducible component of $Y$.
\end{proof}

\begin{thm}\label{tcr}
Let $k$ be an algebraically closed field. Let $U\subseteq X$ be a quasi-smooth
toroidal embedding defined over $k$. Let $\varphi$ be an invertible rational 
function on $X$, let $n\ge 1$ such that $\Char k\nmid n$. Let 
$D=\frac{1}{n}\dv(\varphi)$, and suppose $D|_U$ has integer coefficients.
Let $\pi\colon Y\to X$ be the normalized $n$-th root of $\varphi$. Then 
$\pi^{-1}(U)\subseteq Y$ is a quasi-smooth toroidal embedding,
and $\pi$ is a toroidal morphism.
\end{thm}

\begin{proof} Let $P\in X$ be a point. By Lemma~\ref{bnr}.d), 
we may replace $X$ by an \'etale neighborhood of $P$.
By~\cite[Corollary 2.6]{Ar69}, we may suppose there exists an \'etale morphism 
$$
\tau\colon X\to Z,
$$
where $Z=T_N\emb(\sigma)$ is an affine torus embedding defined over $k$,
the cone $\sigma$ is simplicial, and $U=\tau^{-1}(T_N)$. By~\cite{KKMS}, 
$U\subseteq X$ is a strict toroidal embedding.

Let $D'$ be the part of $D$ which is not supported by $X\setminus U$. It has integer
coefficients, by assumption. By~\cite[Example 5.10]{GS12}, $\Cl \cO_{X,P}$ is 
generated by the primes of $X\setminus U$ passing through $P$. Therefore there exists
$\psi\in k(X)^*$ such that $\dv(\psi)+D'$ is zero on $U$.
Then 
$
\dv(\psi^n\varphi)
$
is supported by $X\setminus U$. By~\cite{KKMS}, there exists $v\in N^*$ such that
$$
\dv(\psi^n\varphi)=\dv(\pi^*\chi^v).
$$
That is $ u\psi^n\varphi =\pi^*\chi^v$ for some unit $u$.
After the \'etale base change $X[\sqrt[n]{u}]\to X$, we may suppose $u=w^n$
for some unit $w$. Therefore 
$$
(w\psi)^n \varphi=\pi^*\chi^v.
$$
By Lemma~\ref{nru}.b), the normalized $n$-th root of $\chi^v$ is a toroidal morphism.
The total space is again quasi-smooth, since $\sigma$ is simplicial. By \'etale base 
change, the normalized $n$-th root of $\pi^*\chi^v$ is also toroidal and quasi-smooth.
By Lemma~\ref{invun}, the normalized $n$-th root of $(w\psi)^n \varphi$ is isomorphic to
the normalized $n$-th root of $\varphi$.
\end{proof}

\subsection{Comparison with roots of sections}

Let $X/k$ be a normal variety. 

Suppose $f\in \Gamma(X,\cO_X)$ does not divide zero.
Then $X[f,n]$ coincides with the normalization of $X[\sqrt[n]{f}]$ (root of regular function).
If $f$ is a unit on $X$, the root is already normal, and therefore $X[f,n]=X[\sqrt[n]{f}]$.

Let $\varphi$ be an invertible rational function on $X$. Let $V=X\setminus \Supp D$ be
the (open dense) locus where $\varphi$ is a unit. The restriction of $X[\varphi,n]\to X$ to $V$ 
coincides with the $n$-th root of the unit $\varphi|_V$. Therefore $X[\varphi,n]$ is obtained
by normalizing $X$ in the function field of each irreducible component of $V[\sqrt[n]{\varphi|_V}]$.

Let $\cL$ be an invertible $\cO_X$-module. Let $s\in \Gamma(X,\cL^n)$ be non-zero.
Choose an open dense subset $U\subseteq X$ such that $\cL|_U$ has a nowhere zero section $u$.
Let $s|_U=\varphi u^n$ with $\varphi\in \Gamma(U,\cO_U)$. Then the normalization
of $X[\sqrt[n]{s}]$ coincides with $X[\varphi,n]$. 


\section{Index one covers of torsion divisors}


Let $k$ be an algebraically closed field. Let $X/k$ be an irreducible, 
normal algebraic variety. Let $D$ be a $\Q$-Weil divisor on $X$ which is {\em torsion}, that is 
$iD\sim 0$ for some $i\ge 1$. The {\em index} of $D$ is 
$$
r=\min\{i\ge 1;iD\sim 0\}.
$$
We suppose $\Char k\nmid r$.
Choose a non-zero rational function $\varphi\in k(X)^\times$ such that $\dv(\varphi)=rD$.

\begin{lem}
The polynomial $T^r-\varphi\in k(X)[T]$ is irreducible.
\end{lem}

\begin{proof} We may apply Lemma~\ref{de}. Suppose $1\le d\mid n$ and 
$T^d-\varphi$ has a root $\psi\in K$. Then $\dv(\psi)=\frac{r}{d}D$, hence
$\frac{r}{d}D\sim 0$. The minimality of $r$ implies $d=1$.
\end{proof}

We deduce that $k(X)[T]/(T^r-\varphi)$ is a field, denoted $k(X)(\sqrt[r]{\varphi})$.
The Kummer field extension 
$$
k(X)\subset k(X)(\sqrt[r]{\varphi})
$$ 
has Galois group $\mu_r$. Let $\psi$ be a root of $T^r-\varphi$ in this extension. 
The Galois group action induces an eigenspace decomposition
$$
k(X)(\sqrt[r]{\varphi})=\oplus_{i=0}^{r-1}k(X)\cdot \psi^i.
$$
Let $\pi\colon Y\to X$ be the normalization of $X$ in the Kummer extension.
By construction, $Y/k$ is an irreducible, normal algebraic variety with quotient 
field $k(X)(\sqrt[r]{\varphi})$. The root $\psi$ identifies with a rational function 
on $Y$ such that $\psi^r=\pi^*\varphi$. In particular, 
$$
\dv(\psi)=\pi^*D.
$$ 
So $\pi^*D$ is linearly trivial on $Y$. The morphism $\pi$ is finite, determined as follows:

\begin{lem} The Galois group $\mu_r$ acts on $Y$ relative to $X$. The eigenspace
decomposition is 
$$
\pi_*\cO_Y=\oplus_{i=0}^{r-1}\cO_X(\lfloor iD\rfloor)\cdot \psi^i.
$$
\end{lem}

\begin{proof} We have $\pi_*\cO_Y=\oplus_{i=0}^{r-1}\cF_i\cdot \psi^i$ for
some subspaces $\cF_i\subset k(X)$. Locally on $X$, a non-zero rational function 
$a\in k(X)^\times$ belongs to $\cF_i$ if and only if $\pi^*a \cdot \psi^i\in \cO_Y$. 
Since $Y$ is normal, this is equivalent to $\dv(\pi^*a)+i\pi^*D\ge 0$, that is 
$\dv(a)+iD\ge 0$, or $a\in \cO_X(\lfloor iD \rfloor)$. Therefore $\cF_i=\cO_X(\lfloor iD \rfloor)$.
\end{proof}

We deduce that $(Y/X,\psi)$ is the normalized $r$-th root of $X$ with respect to $\varphi$.

Call $\pi\colon Y\to X$ the {\em index one cover associated to the torsion $\Q$-divisor $D$}.
It depends on the choice of $\varphi$. If $\varphi_1,\varphi_2$ are two choices, they differ
by a unit $u\in \Gamma(X,\cO_X^\times)$, and the two associated morphisms $Y_i\to X \ (i=1,2)$
become isomorphic after base change with the \'etale covering $X[\sqrt[r]{u}]\to X$.
If $X/k$ is proper, it follows that $Y_i\to X\ (i=1,2)$ are isomorphic, and therefore
$\pi$ does not depend on the choice of $\varphi$.

Let $D'=D+(f)$. Then $D'$ is again torsion, of the same index.
We have $rD'=(\varphi f^r)$ and $(\psi f)^r=\varphi f^r$. Therefore the Kummer
field is the same, so $Y\to X$ is also an index one cover of $D'$.
In conclusion, for two linearly equivalent torsion $\Q$-divisors, one may
choose isomorphic index one covers. In general, any two become isomorphic
after an \'etale base change of $X$ (taking the $r$-th root of some global unit).

Index one covers do not commute with restriction to open subsets, since
the index may drop after restricting to an open subset.



\begin{thebibliography}{BM11}


\bibitem{Ar69}
Artin, M.,
{\em Algebraic approximation of structures over complete local rings.}
{Inst. Hautes \'Etudes Sci. Publ. Math. No. 36 1969, 23�-58.}

\bibitem{BM11A} 
Bierstone, E.; Milman, P.D.,
{\em Resolution except for minimal singularities I.}
{preprint arXiv:1107.5595 (2011).}

\bibitem{Dan78}
Danilov, V. I.,
{\em The geometry of toric varieties.}
{Uspekhi Mat. Nauk 33 (1978), no. 2(200), 85 -- 134.}

\bibitem{Del71}
Deligne P., 
{\em Th\'eorie de Hodge II,}
{Publ. Math. IHES, 40 (1971), 5--57.}

\bibitem{RCII}
Esnault, H.; Viehweg, E.,
{\em Rev\^etements cycliques. II (autour du th\'eor\'eme d'annulation de J. Koll\'ar).}
{G\'eom\'etrie alg\'ebrique et applications, II (La R\'abida, 1984), pp. 81-�96, Travaux en 
Cours, 23, Hermann, Paris, 1987.} 

\bibitem{EVlect}
Esnault, H.; Viehweg, E.,
{\em Lectures on vanishing theorems.} 
{DMV Seminar, 20. Birkh\"auser Verlag, Basel, 1992.}

\bibitem{GS12}
Geraschenko, A.; Satriano, M.,
{\em Torus Quotients as Global Quotients by Finite Groups.}
{preprint arXiv:1201.4807 (2012).}

\bibitem{KKMS}
Kempf, G.; Knudsen, F. F.; Mumford, D.; Saint-Donat, B.,
{\em Toroidal embeddings. I.} 
{Lecture Notes in Mathematics, Vol. 339. Springer-Verlag, Berlin-New York, 1973.} 

\bibitem{Kol95}
Koll\'ar, J.,
{\em Shafarevich maps and automorphic forms.} 
{M. B. Porter Lectures. Princeton University Press, Princeton, NJ, 1995.}

\bibitem{Kol97}
Koll\'ar, J.,
{\em Singularities of pairs.} 
{Algebraic geometry, Santa Cruz 1995, 221 -- 287, 
Proc. Sympos. Pure Math., 62, Part 1, Amer. Math. Soc., Providence, RI, 1997.}

\bibitem{La65}
Lang, S.,
{\em Algebra.}
{Addison-Wesley Publishing Company, 1971.}

\bibitem{Oda88}
Oda, T.,
{\em Convex bodies and algebraic geometry. An introduction 
to the theory of toric varieties.} 
{Ergebnisse der Mathematik und ihrer Grenzgebiete (3) 15. 
Springer-Verlag, Berlin, 1988.}

\bibitem{PS}
Peters, C. A. M.; Steenbrink, J. H. M.,
{\em Mixed Hodge Structures.}
{Ergebnisse der Mathematik {\bf 52} (2008), Springer-Verlag.}

\bibitem{YPG87}
Reid, M.,
{\em Young person's guide to canonical singularities}, 
{Algebraic Geometry (Bowdoin,1985), Proc. Sympos. Pure Math. 
{\bf 46:1}, Amer. Math. Soc., Providence, RI (1987), 345--414.}

\bibitem{Ste76}
Steenbrink, J. H. M.,
{\em Mixed Hodge structure on the vanishing cohomology.} 
{Real and complex singularities (Proc. Ninth Nordic Summer 
School/NAVF Sympos. Math., Oslo, 1976), pp. 525 -- 563. 
Sijthoff and Noordhoff, Alphen aan den Rijn, 1977.} 	

\bibitem{Zar69}
Zariski, O.,
{\em An introduction to the theory of algebraic surfaces.}
{Lect. Notes in Math. {\bf 83} (1969).}

\end{thebibliography}
\end{document}